\documentclass[10pt]{scrartcl}

\usepackage{tikz}
\usetikzlibrary{matrix,arrows,decorations.pathmorphing,shapes.geometric}

\usepackage{etex}
\usepackage{subcaption}
\usepackage[utf8]{inputenc}
\usepackage{a4wide}
\usepackage{graphicx}
\usepackage{wrapfig}
\usepackage[hmarginratio={1:1},     
vmarginratio={1:1},     
textwidth=15cm,        
textheight=22cm,
heightrounded,]{geometry}

\usepackage{amsmath,accents}
\usepackage{amsthm}
\usepackage{amssymb}

\usepackage{mathrsfs, mathtools}
\usepackage{overpic}

\usepackage[utf8]{inputenc}

\usepackage{tikzit}

\tikzstyle{black dot}=[fill=black, draw=black, shape=circle, minimum size=3pt, inner sep=0pt]
\tikzstyle{black dot small}=[fill=black, draw=black, shape=circle, minimum size=3pt, inner sep=0pt]

\tikzstyle{big white circle}=[fill=white, draw=black, shape=circle, minimum width=0.75cm]
\tikzstyle{white dot big}=[fill=white, draw=black, shape=circle, inner sep=1pt]
\tikzstyle{white dot}=[fill=white, draw=black, shape=circle, minimum size=3pt, inner sep=0pt]
\tikzstyle{flat box}=[fill=white, draw=black, shape=rectangle, minimum width=2.5cm, minimum height=0.5cm]
\tikzstyle{square}=[fill=white, draw=black, shape=rectangle]
\tikzstyle{flat box 2}=[fill=white, draw=black, shape=rectangle, minimum height=0.5cm, minimum width=1.0cm]
\tikzstyle{over }=[front]
\tikzstyle{theta}=[fill=black, draw=black, shape=ellipse, minimum height=6pt, minimum width=6pt, inner sep=0pt]
\tikzstyle{thetabig}=[fill=black, draw=black, shape=ellipse, minimum width=1cm, minimum height=0.01cm]
\tikzstyle{thetainv}=[fill=white, draw=black, shape=ellipse, minimum height=6pt, minimum width=6pt, inner sep=0pt]
\tikzstyle{thetabinv}=[fill=white, draw=black, shape=ellipse, minimum width=1cm, minimum height=0.01cm]
\tikzstyle{black over}=[fill=white, draw=black, shape=circle]

\tikzstyle{mid arrow}=[-, postaction={on each segment={mid arrow}}]
\tikzstyle{end arrow}=[->]
\tikzstyle{red mid arrow}=[-, draw={rgb,255: red,214; green,42; black,51}, postaction={on each segment={mid arrow}}, line width=1.1pt]
\tikzstyle{rdots}=[-,dotted, draw=black,line width=1pt]
\tikzstyle{blue}=[-, draw=black, line width=1.1pt]
\tikzstyle{blue mid arrow}=[-, draw=black, postaction={on each segment={mid arrow}}, line width=1.1pt]
\tikzstyle{over}=[-, link]
\tikzstyle{mapsto}=[{|->}]
\tikzstyle{blarrow}=[draw=black, ->, line width=1.1pt]
\tikzstyle{bl}=[-, draw=black, line width=1.1pt]
\tikzstyle{black over}=[-, link2, line width=1.1pt]

\usetikzlibrary{decorations.pathreplacing,decorations.markings}
\tikzset{
  on each segment/.style={
    decorate,
    decoration={
      show path construction,
      moveto code={},
      lineto code={
        \path [#1]
        (\tikzinputsegmentfirst) -- (\tikzinputsegmentlast);
      },
      curveto code={
        \path [#1] (\tikzinputsegmentfirst)
        .. controls
        (\tikzinputsegmentsupporta) and (\tikzinputsegmentsupportb)
        ..
        (\tikzinputsegmentlast);
      },
      closepath code={
        \path [#1]
        (\tikzinputsegmentfirst) -- (\tikzinputsegmentlast);
      },
    },
  },
  mid arrow/.style={postaction={decorate,decoration={
        markings,
        mark=at position .5 with {\arrow[#1]{stealth}}
      }}},
}
\tikzset{%
  link/.style    = { white, double = black, line width =2.0pt,
                     double distance = 1.1pt },
    link2/.style    = { white, double = black, line width = 1.8pt,
  	double distance = 0.4pt },
  channel/.style = { white, double = black, line width = 0.8pt,
                     double distance = 0.6pt },
}

\usepackage[all,cmtip]{xy}

\makeatletter
\DeclareRobustCommand{\em}{%
	\@nomath\em \if b\expandafter\@car\f@series\@nil
	\normalfont \else \slshape \fi}
\makeatother

\usepackage[hidelinks]{hyperref}

\numberwithin{equation}{section}
\usepackage{mathtools}
\mathtoolsset{showonlyrefs}
\numberwithin{equation}{section}

\newtheoremstyle{style1}
{13pt}
{13pt}
{}
{}
{\normalfont\bfseries}
{.}
{.5em}
{}

\theoremstyle{style1}

\newtheorem{definition}{Definition}[section]
\newtheorem{example}[definition]{Example}
\newtheorem{remark}[definition]{Remark}

\makeatletter
\newtheorem*{repd@theorem}{\repd@title}
\newcommand{\newrepdtheorem}[2]{%
	\newenvironment{repd#1}[1]{%
		\def\repd@title{#2 \ref{##1}}%
		\begin{repd@theorem}}%
		{\end{repd@theorem}}}
\makeatother

\newrepdtheorem{definition}{Definition}

\newcommand{\catf}[1]{{\mathsf{#1}}}

\newtheoremstyle{style2}
{13pt}
{13pt}
{\slshape}
{}
{\normalfont\bfseries}
{.}
{.5em}
{}

\theoremstyle{style2}

\makeatletter
\newtheorem*{rep@theorem}{\rep@title}
\newcommand{\newreptheorem}[2]{%
	\newenvironment{rep#1}[1]{%
		\def\rep@title{#2 \ref{##1}}%
		\begin{rep@theorem}}%
		{\end{rep@theorem}}}
\makeatother

\newreptheorem{theorem}{Theorem}
\newreptheorem{corollary}{Corollary}

\newtheorem{lemma}[definition]{Lemma}
\newtheorem{theorem}[definition]{Theorem}
\newtheorem{proposition}[definition]{Proposition}
\newtheorem{corollary}[definition]{Corollary}

\usepackage{tikz}
\usetikzlibrary{matrix,arrows,decorations.pathmorphing,shapes.geometric}
\usepackage{tikz-cd}

\usepackage{enumitem}

\usepackage{needspace}
\newcommand{\spaceplease}{\needspace{5\baselineskip}}
\newcommand{\morespaceplease}{\needspace{20\baselineskip}}

\newcommand{\Z}{\mathbb{Z}}

\newcommand{\ra}[1]{\xrightarrow{\ #1 \ }}

\newcommand{\framed}{\catf{f}E_2}

\usepackage{pifont}

\newcommand{\cat}[1]{\mathcal{#1}}

\newcommand{\Aut}{\operatorname{Aut}}

\newcommand{\Hom}{\operatorname{Hom}}

\newcommand{\id}{\operatorname{id}}

\let\to\undefined
\newcommand{\to}{\longrightarrow}
\let\mapsto\undefined
\newcommand{\mapsto}{\longmapsto}

\newcommand{\Lexf}{\catf{Lex}^\mathsf{f}}

\newcommand{\Vect}{\catf{Vect}}

\newcommand{\opp}{\text{opp}}

\let\colon\undefined\newcommand{\colon}{:}

\DeclareMathSymbol{\Phiit}{\mathalpha}{letters}{"08} 
\DeclareMathSymbol{\Psiit}{\mathalpha}{letters}{"09}
\DeclareMathSymbol{\Sigmait}{\mathalpha}{letters}{"06}
\DeclareMathSymbol{\Xiit}{\mathalpha}{letters}{"04}
\DeclareMathSymbol{\Piit}{\mathalpha}{letters}{"05}\let\Pi\undefined\newcommand{\Pi}{\Piit}
\DeclareMathSymbol{\Gammait}{\mathalpha}{letters}{"00}
\DeclareMathSymbol{\Omegait}{\mathalpha}{letters}{"0A}\let\Omega\undefined\newcommand{\Omega}{\Omegait}
\DeclareMathSymbol{\Upsilonit}{\mathalpha}{letters}{"07}
\DeclareMathSymbol{\Thetait}{\mathalpha}{letters}{"02}
\DeclareMathSymbol{\Lambdait}{\mathalpha}{letters}{"03}\let\Lambda\undefined\newcommand{\Lambda}{\Lambdait}

\let\Phi\undefined\newcommand{\Phi}{\Phiit}
\let\Sigma\undefined\newcommand{\Sigma}{\Sigmait}
\let\Psi\undefined\newcommand{\Psi}{\Psiit}
\let\Gamma\undefined\newcommand{\Gamma}{\Gammait}
\newcommand{\nakar}{\catf{N}^\catf{r}}
\newcommand{\nakal}{\catf{N}^\catf{l}}

\usepackage{multicol}
\newenvironment{pnum}{\begin{enumerate}[label=(\roman*)]}{\end{enumerate}}

\setkomafont{section}{\rmfamily\large}
\setkomafont{subsection}{\rmfamily}
\setkomafont{subsubsection}{\rmfamily}
\setkomafont{subparagraph}{\rmfamily} 

\makeatletter
\renewcommand\section{\@startsection {section}{1}{\z@}%
	{-3.5ex \@plus -1ex \@minus -.2ex}%
	{2.3ex \@plus.2ex}%
	{\normalfont\scshape\centering}}
\makeatother

\usepackage{titletoc}

\dottedcontents{section}[1em]{\rmfamily}{1em}{0.2cm}
\dottedcontents{subsection}[0em]{}{3.3em}{1pc}

\usepackage{titlesec}
\titleformat{\subsection}[runin]
{\normalfont\bfseries}
{\thesubsection}
{0.5em}
{}
[.]

\definecolor{Blue}  {rgb} {0.282352,0.239215,0.803921}
\definecolor{Green} {rgb} {0.133333,0.545098,0.133333}
\definecolor{Red}   {rgb} {0.803921,0.000000,0.000000}
\definecolor{Violet}{rgb} {0.580392,0.000000,0.827450}

\newcounter{jfc}

\usepackage{multicol}

\begin{document}
	
\vspace*{0.5cm}
	\begin{center}	\textbf{\large{The distinguished invertible object \\[0.5ex] as ribbon dualizing object in the Drinfeld center}}\\	\vspace{1cm}	{\large Lukas Müller $^{a}$} \ and \ \ {\large Lukas Woike $^{b}$}\\ 	\vspace{5mm}{\slshape $^a$ Perimeter Institute  \\  N2L 2Y5 Waterloo \\ Canada}	\\[7pt]	{\slshape $^b$ Institut de Mathématiques de Bourgogne\\ 
			UMR 5584 \\ CNRS \& Université de Bourgogne \\ F-21000 Dijon \\  France }\end{center}	\vspace{0.3cm}	
	\begin{abstract}\noindent 
		We prove that the Drinfeld center $Z(\cat{C})$ of a pivotal finite tensor category $\mathcal{C}$ comes with the structure of a ribbon Grothendieck-Verdier category in the sense of Boyarchenko-Drinfeld. Phrased operadically, this makes $Z(\cat{C})$ into a cyclic algebra over the framed $E_2$-operad. The underlying object of the dualizing object is the distinguished invertible object of $\mathcal{C}$ appearing in the well-known Radford isomorphism of Etingof-Nikshych-Ostrik. Up to equivalence, this is the unique ribbon Grothendieck-Verdier structure on $Z(\cat{C})$ extending the canonical balanced braided structure that $Z(\cat{C})$ already comes equipped with. The duality functor of this ribbon Grothendieck-Verdier structure coincides with the rigid duality if and only if $\mathcal{C}$ is spherical in the sense of Douglas-Schommer-Pries-Snyder. The main topological consequence of our algebraic result is that $Z(\cat{C})$ gives rise to an ansular functor, in fact even a modular functor regardless of whether $\mathcal{C}$ is spherical or not. In order to prove the aforementioned uniqueness statement for the ribbon Grothendieck-Verdier structure, we derive a seven-term exact sequence characterizing the space of ribbon Grothendieck-Verdier structures on a balanced braided category. This sequence features the Picard group of the balanced version of the Müger center of the balanced braided category.
		\end{abstract}

\tableofcontents

\spaceplease
\section{Introduction and summary}
The notion of a \emph{finite tensor category}~\cite{etingofostrik,egno}
has, since its introduction almost twenty years ago, developed into one 
of the key algebraic notions that allows us to describe rich classes 
of monoidal categories naturally appearing in quantum topology.
Finite tensor categories include fusion categories, but the notion is designed to accommodate
\emph{non-semisimple} examples as well,	 such as those arising from finite groups in positive characteristic, 
certain quantum groups or logarithmic conformal field theories.  

The definition of a finite tensor category relies on the notion of a \emph{finite linear category}. For an 
algebraically closed field $k$ that we fix throughout, a finite linear category is a linear abelian category with finite-dimensional morphism spaces, enough projective objects, finitely many simple objects up to isomorphism subject to the requirement that every object has finite length. 
A \emph{finite tensor category} 
is then defined as finite linear category
with a bilinear monoidal product that is \emph{rigid}, i.e.\ admits left duals $-^\vee$ and right duals $^\vee-$ 
(where we follow the conventions of \cite[Section~2.10]{egno}); moreover, one 
 requires the monoidal unit to be simple. 
A \emph{pivotal structure} on a finite tensor category $\cat{C}$
 is a monoidal isomorphism $\omega : \id_\cat{C} \to -^{\vee \vee}$
 from the identity of $\cat{C}$ to the square of the dual functor $-^\vee : \cat{C}\to\cat{C}^\opp$ from $\cat{C}$ to its opposite category $\cat{C}^\opp$
  (once we have a pivotal structure, left and right duals coincide; therefore, we will not distinguish between left and right duals in that situation).

 For any finite tensor category $\cat{C}$,
 we can define the so-called \emph{Drinfeld center} $Z(\cat{C})$ as the category of pairs $(X,\beta)$, where $X$ is an object in $\cat{C}$ and $\beta$ is a \emph{half braiding}, 
 i.e.\ an natural isomorphism $\beta :X\otimes-\cong -\otimes X$ subject to the so-called hexagon axioms that ensure that $Z(\cat{C})$ becomes a \emph{braided} finite tensor category.
 A natural, but subtle question is the one about the duality on the Drinfeld center.
The Drinfeld center $Z(\cat{C})$ is rigid, but this duality may have poor properties.
The key problem is the following: The pivotal structure on $\cat{C}$ gives rise to a pivotal structure on $Z(\cat{C})$ and hence to a balancing, i.e.\  a natural automorphism $\theta$ of $\id_{Z(\cat{C})}$
 such that $\theta_I=\id_I$ and
 $\theta_{X\otimes Y}=c_{Y,X}c_{X,Y}(\theta_X \otimes \theta_Y)$ for $X,Y \in Z(\cat{C})$, where $c$ denotes the braiding on $Z(\cat{C})$. We refer to this as the \emph{canonical balanced braided structure} on the Drinfeld center of a pivotal finite tensor category.
 Unfortunately, this will generally not make $Z(\cat{C})$ into a ribbon category because the relation $\theta_X^\vee = \theta_{X^\vee}$ for $X\in Z(\cat{C})$ may not hold. This is 
  the Open~Problem~(7) raised by Müger in \cite[Section~6]{mueger-tensor}, namely to prove
  that (or to investigate when)
  the Drinfeld center $Z(\cat{C})$ of a finite tensor category is a \emph{modular category}, i.e.\ a finite ribbon category with non-degenerate braiding, where the latter means that all objects that trivially double braid with all other objects are necessarily finite direct sums of the monoidal unit.
  The braiding of $Z(\cat{C})$ is automatically non-degenerate by \cite[Proposition~4.4]{eno-d}
  and \cite[Theorem~1.1]{shimizumodular}, so the problem reduces to the question:
  \begin{itemize}
  	\item[(R)]\emph{When (and with which structure) is the Drinfeld center
  		 of a  finite tensor category ribbon?}
  \end{itemize}
      While it may seem like a minor detail whether or not the ribbon property  is fulfilled for the Drinfeld center, the topological implications are significant, namely with respect to the question whether one may construct mapping class group representations from $Z(\cat{C})$; we expand on this in a moment.

      In this article, we give an answer to question~(R): We prove that the Drinfeld center of every pivotal finite tensor category is always ribbon in a natural way if  we allow a \emph{ribbon Grothendieck-Verdier structure} in the sense of Boyarchenko-Drinfeld~\cite{bd} instead of a traditional ribbon structure (note that $\cat{C}$ is still a finite tensor category in the conventional sense, i.e.\ with rigid duality).
      As a special case, we recover an answer to~(R) given by Shimizu~\cite{shimizuribbon} (the semisimple version of which was already given by Müger~\cite{mueger-sph}), as we explain in a moment.

      Before giving a \emph{topological} justification for the
      slight, but critical adjustment of the framework in which we answer question~(R), let us clarify the notion of a ribbon Grothendieck-Verdier category:
      For a monoidal category $\cat{B}$, a \emph{dualizing object}  in the sense of Boyarchenko-Drinfeld~\cite{bd} is an object $K\in \cat{B}$ such that the hom functor $\cat{B}(K,X\otimes -)$ 
      is representable for all $X\in \cat{B}$ by some object $DX$  in a way that the functor $D:\cat{B}\to\cat{B}^\opp$ sending $X$ to $DX$ is an equivalence
      (we should  warn the reader that our conventions are dual to the ones in \cite{bd}). The functor $D:\cat{B}\to\cat{B}^\opp$ is then called \emph{Grothendieck-Verdier duality}.
      It follows directly from the definition that $K\cong DI$.
      A Grothendieck-Verdier duality need not be a rigid duality.
      A \emph{ribbon Grothendieck-Verdier category}
      is a balanced braided category $\cat{B}$ (as recalled above) with a Grothendieck-Verdier 
      duality whose duality functor $D$ satisfies
      $\theta_{DX}=D\theta_X$ for all $X\in \cat{B}$.
      We will consider throughout Grothendieck-Verdier categories in the symmetric monoidal bicategory $\Lexf$ of finite categories, left exact functors and natural transformations, i.e.\ the underlying category is a finite linear category, 
       the monoidal product is a 1-morphism in $\Lexf$ and the natural isomorphisms (braiding, balancing etc.) are linear isomorphisms.
      Considering left exact monoidal products instead of right exact ones might seem unusual, but this is just a matter of passing to the opposite category. Moreover, finite tensor categories have an exact monoidal product in any case because of rigidity.

      In order to build a ribbon Grothendieck-Verdier structure on the Drinfeld center of a pivotal finite tensor category $\cat{C}$, we need a dualizing object.
      The underlying object in $\cat{C}$ of this dualizing object will be the
      \emph{distinguished invertible object} $\alpha \in \cat{C}$
      introduced by Etingof-Nikshych-Ostrik in \cite{eno-d}
      that comes with a monoidal natural isomorphism
      \begin{align}
      \label{radfordisointro} -^{\vee\vee\vee\vee} \cong \alpha \otimes - \otimes \alpha^{-1} \ , 
      \end{align}
      the \emph{Radford isomorphism}.
      
      We may now state the main result:

      \begin{reptheorem}{thmribbon}
      	Let $\cat{C}$ be a pivotal finite tensor category.
      	Then  the distinguished invertible object of $\cat{C}$ equipped with the half braiding induced by the Radford isomorphism and the pivotal structure of $\cat{C}$ is a dualizing object that makes $Z(\cat{C})$ a ribbon Grothendieck-Verdier category.
      	Up to equivalence, this is the only ribbon Grothendieck-Verdier structure on $Z(\cat{C})$ that extends the canonical balanced braided structure.
      \end{reptheorem}

      The fact that the distinguished invertible object, seen as object in the Drinfeld center, is dualizing is an easy observation (that part would be true for any invertible object in the center). As one might expect, it is the `ribbon'  and the uniqueness part of the above statement that is somewhat involved. 
    
    We can say a little more: The \emph{space} of choices for the ribbon Grothendieck-Verdier structure is a $Bk^\times$, the classifying space for the group $k^\times$ of units of $k$, see Corollary~\ref{corcyclic}.
      
      The slight, but critical adjustment of the framework in which we answer question~(R) is justified as follows:
   \begin{pnum}
   	\item
     Even though we change the notion of duality from the traditional rigid duality to a Grothen\-dieck-Verdier duality, we retain 
     control over when exactly
     the ribbon Grothen\-dieck-Verdier structure that we describe is a traditional ribbon structure. This will be the case if and only if
      	$\cat{C}$ is \emph{spherical} in the sense of Douglas-Schommer-Pries-Snyder~\cite{dsps}.
      	 A pivotal finite tensor category $\cat{C}$ is called \emph{spherical}~\cite[Definition~3.5.2]{dsps}
      	if $\alpha \cong I$ in $\cat{C}$ (this property is called \emph{unimodularity})
      	and if the trivialization of $-^{\vee\vee\vee\vee}$ resulting from $\alpha \cong I$ 
      	is equal to the one provided by the square of the pivotal structure.
      	 In the spherical case, our main result will reduce to an answer to question~(R) given by Shimizu in~\cite{shimizuribbon} in the spherical case; the semisimple version of Shimizu's result was already proven by Müger in~\cite[Theorem~1.2]{mueger-sph}.
      	 
      	 In fact, we find:
      	 
      	  \begin{repcorollary}{corspherical}
      	 	For the 
      	 ribbon Grothendieck-Verdier duality on the Drinfeld center $Z(\cat{C})$ of a pivotal finite tensor category $\cat{C}$ described in Theorem~\ref{thmribbon}, the following are equivalent:
      	 \begin{pnum}
      	 	\item The dualizing object is isomorphic to the unit in $Z(\cat{C})$.
      	 	\item $\cat{C}$ is spherical.
      	 	\item The Grothendieck-Verdier duality of $Z(\cat{C})$ agrees with the rigid duality.
      	 	\item $Z(\cat{C})$ with its canonical balanced braided structure is a modular category.
      	 \end{pnum} 
      	 \end{repcorollary}

      	\item  The most important justification for 
      	answering question~(R) in the framework of Grothendieck-Verdier dualities is that
      	 for the question whether the Drinfeld center can be used to build representations of mapping class groups,
      	it will actually be \emph{completely irrelevant} whether we are dealing with a traditional ribbon structure or a ribbon Grothendieck-Verdier duality.  
     This is in fact a rather non-trivial point that we have to explain:
 The main reason why it is interesting to have the structure of a modular category on the Drinfeld center is that a modular category (even if it is non-semisimple) can be used as an input for the Lyubashenko construction~\cite{lyubacmp} to produce a modular functor, i.e.\ mapping class group representations. 
Lyubashenko's construction cannot be applied to a Drinfeld center $Z(\cat{C})$ 
if $Z(\cat{C})$ is not ribbon, but
one can still follow an operadic approach. 
Since $Z(\cat{C})$ is balanced and braided, it is an algebra over the framed $E_2$-operad $\framed$, the operad whose $n$-ary operations are given by embeddings
of $n$ two-dimensional disks into another disk that are composed by translations, rescalings and rotations~\cite{WahlThesis,salvatorewahl}. 
	But the framed $E_2$-operad is also cyclic in the sense of Getzler-Kapranov~\cite{gk}, i.e.\ it comes with a prescription to consistently permute inputs with the output.
	The ribbon Grothendieck-Verdier duality gives us exactly a cyclic structure on a ($\Lexf$-valued) $\framed$-algebra~\cite{cyclic}.  
	A cyclic framed $E_2$-algebra structure, in particular the one that we have on $Z(\cat{C})$, extends uniquely to an \emph{ansular functor}, see~\cite[Section~7.2]{cyclic} and~\cite[Theorem~5.9]{mwansular}, i.e.\ a system of representations of mapping class groups of handlebodies (the handlebodies are always compact oriented and three-dimensional).
	In other words, we obtain for each handlebody $H_g$, characterized by the genus $g\ge 0$ of its boundary surface, a vector space $B_g^{Z(\cat{C})}$,
		the so-called \emph{space of conformal blocks for $H_g$}, together with a representation of the mapping class group of $H_g$. 
	Based on the general description of spaces of conformal blocks in~\cite{cyclic,mwansular}, we give in Corollary~\ref{corconfblocks} a formula for  $B_g^{Z(\cat{C})}$.
For $Z(\cat{C})$, the handlebody group representations on the spaces of conformal blocks will actually extend to representations of mapping class groups of surfaces. In particular, the space of conformal blocks can, in this case, be associated --- without ambiguity --- to a surface instead of a handlebody.
	The latter easily will follow 
	as an application of 
	 the factorization homology approach to modular functors in the paper~\cite{brochierwoike} (however, this is just a background information and not needed for the present paper).

	   The topological perspective provides a characterization of sphericality in terms of the ansular (in fact modular) functor associated to $Z(\cat{C})$, more precisely in terms of the space of conformal blocks $B_0^{Z(\cat{C})}$ associated by the ansular functor of $Z(\cat{C})$ to the three-dimensional ball with boundary or, equivalently, the modular functor of $Z(\cat{C})$ to the sphere. In terms of language, we will follow the second perspective and think of $B_0^{Z(\cat{C})}$ as being associated to the sphere as this is much more common (this change in perspective is also consistent with the fact that in the oriented setting both the mapping class group of the ball and the sphere are trivial). By \cite[Corollary~6.3]{mwansular} $B_0^{Z(\cat{C})}$ is isomorphic to the vector space $Z(\cat{C})(\alpha,I)$ of morphisms
	   from the dualizing object
	   $\alpha \in Z(\cat{C})$ whose half braiding is induced by the Radford isomorphism and the pivotal structure to $I\in Z(\cat{C})$. 
	 
	 \begin{repcorollary}{ValueonS2}
	 	A pivotal finite tensor category is spherical if and only if the space of conformal blocks of the sphere for its Drinfeld center is one-dimensional.
	 \end{repcorollary}
	 
\end{pnum}

Finally, we want to highlight one step in the proof of the uniqueness statement contained in Theorem~\ref{thmribbon} that may be of independent interest, namely the answer to the question how 
 many \emph{other}
 ribbon Grothendieck-Verdier dualities exist on a ribbon Grothendieck-Verdier category $\cat{A}$
  while keeping the underlying
 balanced braided category fixed.
 In other words,
 we would like to understand
 the set
 $\catf{RibGV}(\cat{A})$
 of equivalence classes of ribbon Grothendieck-Verdier structures on the underlying balanced braided category of $\cat{A}$. This is a  pointed set with the class of $\cat{A}$ being the base point.
 It turns out that we can describe $\catf{RibGV}(\cat{A})$, and actually much more than that,
  through 
  an exact sequence featuring the Picard group
 $\catf{Pic}(Z_2^\catf{bal}(\cat{A}))$ of the balanced Müger center
 $Z_2^\catf{bal}(\cat{A})$ of $\cat{A}$. The latter is defined 
 as full subcategory of $\cat{A}$
 spanned by the objects that trivially double braid with every other object and have trivial balancing. 
 This exact sequence is the following (we explain the notation after the statement of the result):
 
 \spaceplease
 \begin{reptheorem}{thmexactseq}
 	For any ribbon Grothendieck-Verdier category $\cat{A}$ in $\Lexf$, there is an exact sequence
 	\begin{center}
 		\begin{equation}\centering
 			\hspace*{-3cm}\begin{tikzcd}
 				1\ar[rrr] &&&
 				\catf{cAut}_{\otimes}(\id_\cat{A}) 	\ar[rrr]
 				&&& \catf{Aut}_{\otimes}(\id_\cat{A})    \ar[out=355,in=175,dllllll]    \\
 				\catf{Aut}_\cat{A}(I) 	 \ar[rrr] &&&  \catf{cAut}_{\framed}(\cat{A}) \ar[rrr] &&& \catf{Aut}_{\framed}(\cat{A}) \ar[out=355,in=175,dllllll] \\ \catf{Pic}(Z_2^\catf{bal}(\cat{A}))   \ar[rrr] &&&   \catf{RibGV}(\cat{A}  ) \ar[rrr] &&& 1
 			\end{tikzcd} 
 	\end{equation}\end{center}
 	of groups, except for the last stage where it is an exact sequence of pointed sets. 
 \end{reptheorem}

 Here $\catf{Aut}_{\framed}(\cat{A})$ is the group of isomorphism classes of balanced braided automorphisms of $\cat{A}$, and $\catf{Aut}_{\otimes}(\id_\cat{A})$ is the group of monoidal isomorphisms of the identity of $\cat{A}$ 
 while $\catf{cAut}_{\framed}(\cat{A})$ and $\catf{cAut}_{\otimes}(\id_\cat{A})$ are cyclic analogs of these groups that are
 additionally compatible with the Grothendieck-Verdier duality, see Section~\ref{seccyclicstructures} for details.

 As an application of Theorem~\ref{thmexactseq}, we prove the following statement which is the key result for proving the uniqueness statement contained in Theorem~\ref{thmribbon}:
 
 \begin{repcorollary}{corcyclic}
 		For a balanced braided category $\cat{A}$ in $\Lexf$ with non-degenerate braiding, there is, up to equivalence, at most one extension to a ribbon Grothendieck-Verdier structure. More precisely, the space of choices is empty or $B\Aut_\cat{A}(I)$. The latter is $Bk^\times$ if the unit of $\cat{A}$ is simple.	
 \end{repcorollary}
	
	\vspace*{0.2cm}\textsc{Acknowledgments.} We thank
	Adrien Brochier, 	
	Peter Schauenburg,	
	Christoph Schweig\-ert and
	Nathalie Wahl 
	for  helpful discussions related to this project, and also 
	the anonymous referee whose comments have helped us improve the manuscript.
	LM gratefully acknowledges support by the Max Planck Institute for Mathematics in Bonn,
	 where part of this work was carried out,
	  and the Simons Collaboration on Global Categorical Symmetries. Research at Perimeter Institute is supported in part by the Government of Canada through the Department of Innovation, Science and Economic Development and by the Province of Ontario through the Ministry of Colleges and Universities. The Perimeter Institute is in the Haldimand Tract, land promised to the Six Nations.
	LW gratefully acknowledges support by 
the Danish National Research Foundation through the Copenhagen Centre for Geometry
and Topology (DNRF151)
at K\o benhavns Universitet during the period in which this project was initiated. 
LW also gratefully acknowledges support
 by the ANR project CPJ n°ANR-22-CPJ1-0001-01 at the Institut de Mathématiques de Bourgogne (IMB).
The IMB receives support from the EIPHI Graduate School (contract ANR-17-EURE-0002).

\section{The ribbon Grothendieck-Verdier structure on the Drinfeld center of a pivotal finite tensor category}
For an algebraically closed field $k$ that we fix throughout, let $\cat{C}$ be a finite tensor category over $k$
	 with pivotal structure, i.e.\ a monoidal isomorphism $\omega : \id_\cat{C} \to -^{\vee \vee}$. 
Then the Drinfeld center $Z(\cat{C})$ is 
 a finite tensor category as well, and it   comes again 
with a pivotal structure. By abuse of notation we denote it again by $\omega$.
Just like any braided category, $Z(\cat{C})$ comes with the \emph{Joyal-Street equivalence} $J:Z(\cat{C})\to Z(\cat{C})$ \cite{joyalstreet2} which is the monoidal functor whose underlying functor is the identity with
the monoidal structure 
 given by the double braiding. 
A balancing on $Z(\cat{C})$ is a  monoidal isomorphism $\id_{Z(\cat{C})} \cong J$ (it is easy to see that this amounts exactly to the definition of a balancing in the introduction).
As any pivotal braided category, $Z(\cat{C})$ comes with a canonical balancing; we will recall the formula for the balancing in~\eqref{eqnpivrigid}.

We will now proceed with some more background on the distinguished invertible object $\alpha \in \cat{C}$ of a finite tensor category $\cat{C}$ and the Radford isomorphism from~\cite[Theorem~3.3]{eno-d}
\begin{align}
\label{radfordiso} -^{\vee\vee\vee\vee} \cong \alpha \otimes - \otimes \alpha^{-1} \ , 
\end{align}
the monoidal isomorphism
that was mentioned in the introduction.
The Radford isomorphism gives rise to a variety of 
monoidal isomorphisms which are important in what follows. Using that $\alpha$ is invertible we
get monoidal isomorphisms  
\begin{align}
\delta^+ \colon \alpha \otimes -   \cong  -^{\vee \vee \vee \vee } \otimes \alpha 
\end{align}
and 
\begin{align}
\delta^- \colon - \otimes \alpha^{-1} \cong \alpha^{-1} \otimes -^{\vee \vee \vee \vee } \ \ .
\end{align}
A finite tensor category with $\alpha\cong I$ is called \emph{unimodular}.
If $\cat{C}$ is unimodular and pivotal, we obtain two trivializations
of $-^{\vee\vee\vee\vee}$, namely one from $\alpha \cong I$ and one from the square of the pivotal structure.
If these agree, the unimodular pivotal category $\cat{C}$ is called \emph{spherical} \cite[Definition~3.5.2]{dsps}.
As explained in \cite[Section~3.5.2]{dsps}, this notion of sphericality is related to the sphericality in terms of quantum traces only in the semisimple case. Beyond semisimplicity, it is an entirely different notion. 
Most remarkably, sphericality, even in the non-semisimple situation, can be perfectly characterized through traces if one uses modified traces instead of quantum traces \cite[Theorem~1.3]{shibatashimizu}.

There is a very profitable perspective on the Radford 
isomorphism based on tools developed by Fuchs-Schaumann-Schweigert
in \cite{fss} 
 that we will heavily rely on: For any finite linear category $\cat{C}$, one can define the \emph{right Nakayama functor} $\nakar : \cat{C}\to\cat{C}$
  via a coend $\nakar := \int^{X\in\cat{C}} \cat{C}(-,X)^* \otimes X$ and the \emph{left Nakayama functor} $\nakal:\cat{C}\to\cat{C}$
   via an end $\nakal := \int_{X\in \cat{C}} \cat{C}(X,-) \otimes X$. If $\cat{C}$ is a finite tensor category, then these (co)ends can be expressed through the distinguished invertible object and the double dual functor. More precisely, we have canonical natural isomorphisms \cite[Lemma~4.10]{fss}
\begin{align}
	\nakar \cong \alpha^{-1} \otimes -^{\vee\vee} \ , \quad \nakal \cong -^{\vee \vee} \otimes \alpha \ , \label{eqnnakafunctorsftc}
	\end{align}
where $\alpha := \nakal I$ is invertible with inverse $\alpha^{-1} = \nakar I$. Here we use the same notation `$\alpha$'
and `$\alpha^{-1}$' that we have used for the 
distinguished object and its inverse, respectively, 
because these agree indeed~\cite[Remark~4.13]{fss}. 
This allows for a very elegant derivation of the Radford isomorphism~\eqref{radfordiso}:
By \cite[Corollary~4.7]{fss} the Nakayama functors $\nakar$ and $\nakal$ of a finite tensor category
 are weakly inverse equivalences, and we agree to equip them with natural isomorphisms
$\nakal \nakar \cong \id_\cat{C}$ and $\nakar \nakal \cong \id_\cat{C}$ making them \emph{adjoint} equivalences $\nakar \dashv \nakal$. 
Now~\eqref{eqnnakafunctorsftc} gives us the natural isomorphism
\begin{align}
\delta: 	\id_\cat{C} \cong \alpha^{-1} \otimes -^{\vee\vee\vee\vee} \otimes \alpha\label{eqndelta}
	\end{align}
that can also be written in the form~\eqref{radfordiso}.

Combining the pivotal structure with the Radford isomorphism allows us to consider $\alpha$
as an element of $Z(\cat{C})$:

\begin{lemma}\label{lemmahalfbraiding}
	For a pivotal finite tensor category $\cat{C}$,
	the natural isomorphism $\sigma : \alpha \otimes -\cong - \otimes \alpha$
	whose component at $X\in\cat{C}$ is defined by 
		\begin{equation}\label{defhalfbraiding}
\sigma_X \colon \alpha \otimes X \ra{\delta^+_X} X^{\vee\vee\vee\vee}  \otimes  \alpha \ra{{\omega_X^2}^{-1}} X\otimes \alpha 
	\end{equation}
	equips the distinguished invertible object
	 $\alpha$ with a half braiding, i.e.\ $\alpha$ can be seen as object in the Drinfeld center $Z(\cat{C})$.
	With this half braiding, we have $\alpha \cong I$ in $Z(\cat{C})$ if and only if $\cat{C}$ is spherical. In a similar way, $\delta^-$ can be used to equip $\alpha^{-1}$ with a half braiding. These
	two objects in $Z(\cat{C})$ are inverse to each other
	with respect to the monoidal product 
	in $Z(\cat{C})$.
	\end{lemma}

In the second step in~\eqref{defhalfbraiding}, the second isomorphism is of course not just ${\omega_X^2}^{-1}$, but it is tensored with an identity, in this case of $\alpha$.  
This is suppressed here and in similar cases below for readability.

\begin{proof}
	Since $\delta^+$ and $\omega$ are natural, so is $\sigma$.
	The monoidality of the transformations $\delta^+$ and $\omega$ tells us that $\sigma$ is in fact a half braiding.

By slightly rephrasing the definition of sphericality, $\cat{C}$ is spherical if and only if $\alpha \cong I$ in $\cat{C}$, and if under this isomorphism, $\sigma$ coincides with the unit constraint of the monoidal unit. But this is equivalent to $\alpha \cong I$ in $Z(\cat{C})$, where $\alpha$ is equipped with the half braiding $\sigma$.  

The remark on $\alpha^{-1}$ with its half braiding being $\otimes$-inverse to $\alpha$ in $Z(\cat{C})$ is clear.
\end{proof}

\begin{example}\label{exgraded}	
	Let $G$ be a finite group and $\lambda$ a normalized 3-cocycle on $G$ with values in $k^\times$ (the group of units of $k$), and let $\Vect_{G}^\lambda$
	be the category of $G$-graded finite-dimensional vector spaces with the usual monoidal structure and the associator 
	that on homogeneous objects $V_g$ (which are given by a finite-dimensional vector space $V$ in degree $g\in G$) 
	is given by
	\begin{align}
		\alpha \colon (V_g\otimes V_{g'})\otimes V_{g''} & \longrightarrow V_g \otimes (V_{g'}\otimes V_{g''}) \\
		(v_g\otimes v_{g'})\otimes v_{g''} &  \longmapsto \lambda(g,g',g'') \cdot v_g \otimes (v_{g'}\otimes v_{g''})
	\end{align} We refer to \cite[Example~2.3.8]{egno} for a textbook reference for this example.
	The dual of an object $V_g$ is given by $V^*_{g^{-1}}$. We identify the double dual functor with the 
	identity using the canonical isomorphism $V^{**}\cong V$ between a finite-dimensional vector space $V$ and its bidual. 
	Pivotal structures on $\Vect_{G}^\lambda$ are
	therefore monoidal automorphisms of the identity functor.
	They are equivalent to   	group homomorphisms $d \colon G \to k^\times$; this follows essentially from~\cite[Example~1.7.3]{TV}.
The  pivotal structure corresponding to a group morphisms $d: G \to k^\times$
	is explicitly given by
	\begin{align}
		\omega_{V_g}^{(d)} \colon V_g & \to V_g \\
		v & \longmapsto d(g)\cdot \lambda(g,g^{-1},g) \cdot v \ \ . \label{eqnpivotalstructurefromd}
	\end{align} 
	We denote by $\Vect_{G}^{\lambda,d}$ the finite tensor category $\Vect_{G}^\lambda$ with the pivotal structure defined by a group morphism
$d: G\to k^\times$ via~\eqref{eqnpivotalstructurefromd}.
	For $\Vect_{G}^\lambda$ the distinguished invertible object and Radford isomorphism are trivial, i.e.\  $\alpha=I=k_{e}$, where $e\in G$ is the neutral element. In fact, the distinguished invertible object coincides with the monoidal unit for every semisimple finite tensor category \cite[Corollary~6.4]{eno-d}.
	
	An element in the Drinfeld center $Z(\Vect_{G}^{\lambda,d})$ consists of a $G$-graded vector
	space $V=\bigoplus_{g\in G} V_g$ together with a half braiding $\sigma_{V,-}$ which is completely determined by
	its components on simple objects
	\begin{align}
	\sigma_{V,h}\colon \bigoplus_{g\in G} V_g \otimes k_h   \longrightarrow \bigoplus_{g\in G} k_h \otimes V_g \ \ .
	\end{align} 
    The isomorphism $\sigma_{V,h}$ can be described by a family
    of isomorphisms $\sigma_{V,h}^g \colon V_g \longrightarrow V_{h^{-1}gh}$ satisfying 
    \begin{align}\label{sigmavhheqn}
    \sigma_{V,hh'}^g=  \lambda(h,h',(hh')^{-1}ghh')^{-1} \lambda(h,h^{-1}gh,h') \lambda(g,h,h')^{-1} \cdot \sigma_{V,h'}^{h^{-1}gh} \circ \sigma_{V,h}^g \ \ .
    \end{align} 	
In the case that $\lambda$ is trivial,
 this condition is equivalent to the maps $\sigma_{V,h}$ defining a representation
of $G$ on $V$. Objects of this type are also known as $G$-Yetter-Drinfeld modules~\cite[Chapter~IX]{kassel}. 
	The object $(\alpha,\sigma) \in Z(\Vect_{G}^{\lambda,d})$, with $\sigma$ being the half braiding from Lemma~\ref{lemmahalfbraiding}, 
	has as underlying object the monoidal unit $I$,
	and the only non-zero components of the half braiding are 
	$\sigma_{I,h}^{e}= (\omega^{(d)}_{k_h})^2$, 
	where $h\in G$ and 
	$(\omega^{(d)})^2$ is the square of the pivotal structure associated to $d:G\to k^\times$. 
	In case $\lambda$ is trivial,
	we conclude that $\Vect_{G}^{\lambda,d}$ with the pivotal structure coming from $d$ is spherical if and only if $d^2$ is trivial.
\end{example}

\begin{lemma}\label{lemmachi}
	Let $F:\cat{C}\to\cat{C}$ be a monoidal autoequivalence of a finite tensor category $\cat{C}$, then the structure isomorphism $\phi_0^F : I\to FI$ of $F$ (that is part of the data of the monoidal functor $F$)
	 induces isomorphisms $\chi^F: \alpha \to F\alpha$
	 and $\bar \chi^F : \alpha^{-1} \to F\alpha^{-1}$
	  for the distinguished invertible object $\alpha$ of $\cat{C}$. 
	\end{lemma}

\begin{proof}
	Let $F^{-1}$ be the weak inverse of $F$ (just like $F$, it is a monoidal autoequivalence of $\cat{C}$). We can arrange $F^{-1}$ to be left adjoint to $F$. Clearly, $F$ and $F^{-1}$ are both exact because they are equivalences.
	Now \cite[Theorem~3.18~(i)]{fss} gives us a canonical isomorphism
	$\zeta : \nakal F \cong F \nakal$ (this uses the adjunction). Evaluation at $I$ yields a canonical isomorphism
	\begin{align}
		\zeta_I: \nakal FI \ra{\cong} F\alpha \ . 
		\end{align}
	Before proceeding with the proof, let us extract from the proof of~\cite[Theorem~3.18~(i)]{fss} where this isomorphism comes from:
	There is a canonical isomorphism
	\begin{align}
		F\alpha \cong   \int_{X \in \cat{C}} \cat{C}(X,I)\otimes FX \cong    \int_{Y\in\cat{C}} \cat{C}(Y,FI) \otimes Y  \ ,     \label{eqnforchi}
	\end{align}
	where
	\begin{itemize}
		\item for the first isomorphism, we use that $F$, by virtue of being an equivalence and therefore in particular exact, preserves the end and, by linearity, the tensoring with a vector space,
		\item and for the second isomorphism, we perform a variable substitution $Y=FX$ in the end.
	\end{itemize}
	The left hand side of~\eqref{eqnforchi} is $F \nakal I$ while the right hand side is $\nakal FI$.
	An inspection of the proof of~\cite[Theorem~3.18~(i)]{fss} (which simplifies of course if $F$ is an equivalence) reveals that this isomorphism is the component $\zeta_I^{-1}: F\alpha = F \nakal I\cong \nakal FI$.
	
	With the isomorphism $\zeta$ at hand, $\chi^F$ can be obtained as
	\begin{align}\chi^F:\alpha = \nakal I 
		\stackrel{\nakal \phi_0   }{\cong} \nakal FI \stackrel{\zeta_I   }{\cong} F \nakal I = F\alpha \ .   \label{eqnchiF} \end{align}
	By replacing $\nakal$ with $\nakar$ we obtain a similar isomorphism for $\alpha^{-1}$ instead of $\alpha$. 
	\end{proof}

\begin{lemma}\label{lemmachi2}
	Let $F:\cat{C}\to\cat{C}$ be a monoidal autoequivalence of a finite tensor category $\cat{C}$ and $\lambda : \id_\cat{C}\to F$ a monoidal isomorphism.
	Then $\lambda_\alpha = \chi^F : \alpha \to F\alpha$
	and $\lambda_{\alpha^{-1}} = \bar \chi^F:\alpha^{-1}\to F\alpha^{-1}$,
	where $\chi^F$ and $\bar \chi^F$ are from Lemma~\ref{lemmachi}.
	\end{lemma}

\begin{proof}
	When postcomposed with 
	 the isomorphism $F\alpha\cong\int_ {X\in\cat{C}}\cat{C}(X,I)\otimes FX$ (this is the first isomorphism in~\eqref{eqnforchi}), the morphism $\lambda_\alpha:\alpha \to F\alpha$ is given by
	\begin{align}
		 \int_{X\in\cat{C}} \cat{C}(X,I)\otimes X \ra{\int_{X\in\cat{C}}\cat{C}(X,I) \otimes \lambda_X }   \int_{X\in\cat{C}} \cat{C}(X,I)\otimes FX \ .     \label{eqnlambdaunderend}
		\end{align}
	This is a consequence of the  naturality of $\lambda$. 
	After the variable substitution $Y=FX$ for the end, we obtain an isomorphism
	 $\int_{X\in\cat{C}} \cat{C}(X,I)\otimes FX\cong \int_{Y\in\cat{C}} \cat{C}(Y,FI)\otimes Y$ (this is the second isomorphism in~\eqref{eqnforchi}).
	 After postcomposing~\eqref{eqnlambdaunderend} with this isomorphism, 
	 we obtain, again by naturality and the universal property of the end,
	 the map 
	 $\nakal(\lambda_I):\int_{X\in\cat{C}} \cat{C}(X,I)\otimes X \cong \int_{X\in\cat{C}} \cat{C}(X,FI)\otimes X$ induced by 
	  $\lambda_I:I\to FI$.
	  But the two maps that we have postcomposed $\lambda_\alpha$ with combine into $\zeta_I^{-1}: F\alpha = F \nakal I\cong \nakal FI$ as was explained in the proof of Lemma~\ref{lemmachi}. Hence, we conclude so far that
	  $\zeta_I^{-1} \circ \lambda_\alpha = \nakal(\lambda_I)$ or, equivalently, $\lambda_\alpha = \zeta_I \circ \nakal(\lambda_I)$.
	  Moreover, $\lambda_I= 
	  \phi_0^F:I\to FI$ by the assumption that $\lambda$ is monoidal. 
	  This leaves us with
	  \begin{align}\lambda_\alpha = \zeta_I \circ \nakal(\phi_0^F) \stackrel{\eqref{eqnchiF}}{\cong} \chi^F \ . 
	  	\end{align} A similar argument applies to $\bar \chi^F$.
	\end{proof}

In order to state the next result, we need the following notion: A \emph{pivotal structure}~\cite[Section~6, in particular Proposition~6.7]{bd} 
on a monoidal category $\cat{B}$ with Grothendieck-Verdier duality $D$ is a 
monoidal isomorphism $\xi :  \id_\cat{B} \to D^2$ such that the component 
$\xi_I : I \to D^2 I$ agrees with the canonical isomorphism $I\cong D^2I$ from \cite[Remark~2.1~(4)]{bd}.

\begin{proposition}[]\label{proppivotal}
	The Drinfeld center $Z(\cat{C})$ of a pivotal finite tensor category $\cat{C}$
	 comes with a Grothendieck-Verdier duality $D=  -^\vee \otimes \alpha$. This Grothendieck-Verdier duality comes with a pivotal structure that at $X\in Z(\cat{C})$ is given by
	\begin{align}
\xi_X: 	
X \ra{\delta_X\ \text{from}\ \eqref{eqndelta}} \alpha^{-1} \otimes X^{\vee \vee \vee \vee} \otimes \alpha \ra{\omega_{{X^{\vee\vee}}}^{-1}}   \alpha^{-1} \otimes X^{\vee \vee} \otimes \alpha =D^2 X \ . 
\end{align}
	\end{proposition}

\begin{remark}\label{remformulaxi}
Before we prove Proposition~\ref{proppivotal},
 we give a useful reformulation of the pivotal structure in terms of the half braiding $\sigma$
  for $\alpha$
 given in Lemma~\ref{lemmahalfbraiding}:
 With the formula for the half braiding, we can see directly that
 $\xi_X$ is equal to the composition
\begin{align}
 X \ra{\text{coevaluation}} \alpha^{-1}\otimes \alpha \otimes X \ra{\sigma_X} \alpha^{-1}\otimes X
\otimes \alpha \ra{\omega_X} \alpha^{-1} \otimes X^{\vee \vee}\otimes\alpha =D^2 X \ .
\end{align}
In the graphical calculus for $\cat{C}$,
 the pivotal structure for the Grothendieck-Verdier duality on $Z(\cat{C})$
 is given by 
 \begin{align}
 	\begin{tikzpicture}
	\begin{pgfonlayer}{nodelayer}
		\node [style=none] (0) at (-19, 9.25) {};
		\node [style=none] (1) at (-15.75, 9.25) {};
		\node [style=none] (4) at (-16.5, 5.75) {};
		\node [style=square] (5) at (-17.5, 8.25) {$\omega$};
		\node [style=none] (6) at (-17.5, 9.25) {};
		\node [style=none] (7) at (-16.5, 5) {};
		\node [style=none] (8) at (-15.5, 7.25) {$\alpha$};
		\node [style=none] (10) at (-16.5, 4.5) {$X$};
		\node [style=none] (12) at (-17.5, 9.75) {$X^{\vee \vee}$};
		\node [style=none] (13) at (-17.75, 6) {};
		\node [style=none] (15) at (-17.25, 6) {};
	\end{pgfonlayer}
	\begin{pgfonlayer}{edgelayer}
		\draw (6.center) to (5);
		\draw [in=90, out=-90] (5) to (4.center);
		\draw (4.center) to (7.center);
		\draw [style=blarrow] (13.center) to (15.center);
		\draw [style=bl, in=180, out=-90] (0.center) to (13.center);
		\draw [style=over, in=-90, out=0] (15.center) to (1.center);
	\end{pgfonlayer}
\end{tikzpicture} \ , \label{eqnpivotalexpl}
 \end{align}
where the crossing denotes the half braiding for $\alpha$.
The graphical calculus will be used throughout the paper. Let  us just recall the main points here: The graphical calculus is used to write 
morphisms in monoidal categories (possibly with more structure and properties).
We write them from bottom to top.
Objects are denoted by lines (for $\alpha$ and $\alpha^{-1}$, we will print this line in fat), and the juxtaposition of lines corresponds to the monoidal product. Evaluation and coevaluation are written as a cup and cap, respectively.
For the lines,
we often use arrows to make the labeling 
more efficient: 
An upwards pointing line labeled with $Y$ represents an object 
$Y\in \cat{C}$,
while a downwards pointing line represents the dual of $Y$. For example, in~\eqref{eqnpivotalexpl} the target object is $\alpha^{-1} \otimes X^{\vee \vee} \otimes \alpha$. 
We refer to~\cite{kassel,TV}
for a textbook introduction
to the graphical calculus.
\end{remark}

\begin{proof}[\slshape Proof of Proposition~\ref{proppivotal}]
The fact that the functor	
$D=  -^\vee \otimes \alpha$ is a Grothendieck-Verdier duality is purely formal: 
Since $Z(\cat{C})$ is rigid, we have  isomorphisms
	\begin{align} 
		Z(\cat{C})(\alpha,X\otimes Y)\cong Z(\cat{C})( X^\vee \otimes \alpha  , Y) 
	\end{align}
natural in $X,Y \in Z(\cat{C})$.
This shows that the equivalence $D=  -^\vee \otimes \alpha$ is a Grothendieck-Verdier duality on $Z(\cat{C})$. 

The non-trivial point is that $\xi$ is a pivotal structure for the Grothendieck-Verdier duality:
\begin{itemize}
	\item
First we  prove that $\xi$ is a monoidal isomorphism $\id_{Z(\cat{C})} \to D^2$ in $Z(\cat{C})$: 
The components $\xi_X$ are in fact morphisms in $Z(\cat{C})$ which follows from the naturality of $\omega$ and 
 the fact that the half braiding of $\alpha$ defines a morphism in $Z(\cat{C})$. 
In order to see that $\xi$ is monoidal, 
let us observe that
 the structure isomorphism
\small
\begin{align}
\mu_{X,Y}:D^2 X \otimes D^2 Y=\alpha \otimes X^{\vee \vee} \otimes \alpha^{-1} \otimes \alpha \otimes Y^{\vee\vee}\otimes \alpha^{-1} \ra{\cong} \alpha \otimes X^{\vee \vee} \otimes  Y^{\vee\vee}\otimes \alpha^{-1} \cong D^2(X\otimes Y)
\end{align}\normalsize
making $D^2$ monoidal by~\cite[Proposition~5.2]{bd} is the isomorphism
 induced by the evaluation of the invertible object $\alpha$.
Let us justify this:  
Recall from (the dual of) \cite[Equation~(2.3)]{bd}
the natural isomorphism \begin{align} q_{X,Y} : Z(\cat{C})(K,X\otimes Y) \to Z(\cat{C})(K,Y\otimes D^2 X) \end{align}
that, when specialized to our particular Grothendieck-Verdier duality at hand, is given by
\begin{align}
	\begin{tikzpicture}
	\begin{pgfonlayer}{nodelayer}
		\node [style=flat box 2] (5) at (-17.5, 7.75) {$f$};
		\node [style=none] (8) at (-17.5, 5.5) {$\alpha$};
		\node [style=none] (10) at (-18, 10.5) {$X$};
		\node [style=none] (16) at (-18, 8) {};
		\node [style=none] (17) at (-17, 8) {};
		\node [style=none] (18) at (-18, 10) {};
		\node [style=none] (19) at (-17, 10) {};
		\node [style=none] (20) at (-17.5, 7.5) {};
		\node [style=none] (21) at (-17.5, 6.75) {};
		\node [style=none] (23) at (-17, 10.5) {$Y$};
		\node [style=none] (25) at (-15.25, 8) {$\longmapsto$};
		\node [style=flat box 2] (26) at (-12.75, 7.75) {$f$};
		\node [style=none] (29) at (-13, 8) {};
		\node [style=none] (30) at (-12.5, 8) {};
		\node [style=none] (31) at (-13.5, 9.25) {};
		\node [style=none] (32) at (-12.5, 10) {};
		\node [style=none] (33) at (-12.75, 7.5) {};
		\node [style=none] (34) at (-12.25, 6.25) {};
		\node [style=none] (35) at (-14, 7.25) {};
		\node [style=none] (38) at (-11.25, 10) {};
		\node [style=none] (39) at (-12.5, 5.75) {};
		\node [style=none] (40) at (-10.75, 10) {};
		\node [style=none] (41) at (-10.25, 5.75) {};
		\node [style=none] (42) at (-10.25, 8) {};
		\node [style=none] (43) at (-10.25, 10) {};
		\node [style=none] (44) at (-17.5, 6) {};
		\node [style=none] (45) at (-11.25, 9.5) {};
		\node [style=none] (46) at (-11.25, 9.25) {};
	\end{pgfonlayer}
	\begin{pgfonlayer}{edgelayer}
		\draw (18.center) to (16.center);
		\draw (19.center) to (17.center);
		\draw [style=bl] (20.center) to (21.center);
		\draw [in=90, out=0] (31.center) to (29.center);
		\draw [style=blue, in=-180, out=-90, looseness=0.75] (33.center) to (34.center);
		\draw [in=90, out=-180, looseness=0.75] (31.center) to (35.center);
		\draw [style=bl, in=0, out=-90, looseness=0.75] (38.center) to (34.center);
		\draw (30.center) to (32.center);
		\draw [in=-180, out=-90, looseness=1.25] (35.center) to (39.center);
		\draw [in=-90, out=0] (39.center) to (40.center);
		\draw [style=blarrow] (41.center) to (42.center);
		\draw [style=bl] (42.center) to (43.center);
		\draw [style=blarrow] (44.center) to (21.center);
		\draw [style=blarrow] (45.center) to (46.center);
	\end{pgfonlayer}
\end{tikzpicture} \ . \label{eq23}
\end{align}
With~\eqref{eq23}, it follows directly that the map~$(*)$ defined by the commutativity of the diagram
	\begin{equation}
	\begin{tikzcd}
		Z(\cat{C})(K, Y \otimes   D^2X_1 \otimes D^2 X_2 )  \ar[dd,swap,"      (*)    "]   &&& Z(\cat{C})(K, X_2 \otimes Y  \otimes   D^2X_1 ) \ar[lll,swap,"q_{X_2,Y\otimes D^2 X_1}"]   \\  \\
		Z(\cat{C})(K, Y \otimes D^2(X_1  \otimes X_2 ) )  &&&      Z(\cat{C})(K,X_1  \otimes X_2 \otimes  Y ) \ar[uu,swap,"q_{X_1,X_2\otimes Y}"] \ar[lll,"q_{X_1\otimes X_2,Y}"]
	\end{tikzcd}
\end{equation}
is the isomorphism induced by the evaluation of $\alpha$. 
But by the construction of the monoidal structure on $D^2$ in \cite[Section~5]{bd} the map~$(*)$ is exactly what gives us the structure maps $\mu_{X_1,X_2}:D^2 X_1 \otimes D^2 X_2 \to D^2 (X_1 \otimes X_2)$; more precisely, $(*)$ is the 
composition with $\id_Y \otimes \mu_{X_1,X_2}$.  
This shows that $\mu$ is, as claimed above, indeed induced by the evaluation of $\alpha$.

Now the equality
\begin{align}
\mu_{X,Y} (\xi_X \otimes \xi_Y)=\xi_{X\otimes Y}
\end{align}
is a consequence of the following computation using the graphical calculus (we use for this Remark~\ref{remformulaxi}): 
\begin{align}
	\begin{tikzpicture}
	\begin{pgfonlayer}{nodelayer}
		\node [style=none] (0) at (-19, 9.25) {};
		\node [style=none] (1) at (-15.75, 8) {};
		\node [style=none] (4) at (-16.5, 5.75) {};
		\node [style=square] (5) at (-17.5, 8.25) {$\omega$};
		\node [style=none] (6) at (-17.5, 9.25) {};
		\node [style=none] (7) at (-16.5, 5) {};
		\node [style=none] (8) at (-19.5, 7.5) {$\alpha$};
		\node [style=none] (10) at (-16.5, 4.5) {$X$};
		\node [style=none] (12) at (-17.5, 9.75) {$X^{\vee \vee}$};
		\node [style=none] (13) at (-18, 6) {};
		\node [style=none] (15) at (-17.75, 6) {};
		\node [style=none] (16) at (-15.75, 8) {};
		\node [style=none] (17) at (-12.5, 9.25) {};
		\node [style=none] (18) at (-13.25, 5.75) {};
		\node [style=square] (19) at (-14.25, 8.25) {$\omega$};
		\node [style=none] (20) at (-14.25, 9.25) {};
		\node [style=none] (21) at (-13.25, 5) {};
		\node [style=none] (22) at (-12, 8) {$\alpha$};
		\node [style=none] (23) at (-13.25, 4.5) {$Y$};
		\node [style=none] (24) at (-14.25, 9.75) {$Y{^{\vee \vee}}$};
		\node [style=none] (25) at (-14, 6) {};
		\node [style=none] (26) at (-13.75, 6) {};
		\node [style=none] (28) at (-11, 7.5) {$=$};
		\node [style=none] (30) at (-9.5, 9.25) {};
		\node [style=none] (31) at (-6.25, 5.75) {};
		\node [style=none] (32) at (-7, 5.75) {};
		\node [style=square] (33) at (-8, 8.25) {$\omega$};
		\node [style=none] (34) at (-8, 9.25) {};
		\node [style=none] (35) at (-7, 5) {};
		\node [style=none] (36) at (-10, 8) {$\alpha$};
		\node [style=none] (37) at (-7, 4.5) {$X$};
		\node [style=none] (38) at (-8, 9.75) {$X^{\vee \vee}$};
		\node [style=none] (39) at (-8.5, 5.25) {};
		\node [style=none] (40) at (-8, 5.25) {};
		\node [style=none] (41) at (-6.25, 5.75) {};
		\node [style=none] (42) at (-3.75, 9.25) {};
		\node [style=none] (43) at (-4.5, 5.75) {};
		\node [style=square] (44) at (-5.5, 8.25) {$\omega$};
		\node [style=none] (45) at (-5.5, 9.25) {};
		\node [style=none] (46) at (-4.5, 5) {};
		\node [style=none] (47) at (-3.5, 7.25) {$\alpha$};
		\node [style=none] (48) at (-4.5, 4.5) {$Y$};
		\node [style=none] (49) at (-5.5, 9.75) {$Y{^{\vee \vee}}$};
		\node [style=none] (50) at (-5.25, 6) {};
		\node [style=none] (51) at (-5, 6) {};
		\node [style=none] (52) at (-1.25, 9.25) {};
		\node [style=none] (53) at (3.75, 9.25) {};
		\node [style=none] (54) at (1.75, 5.75) {};
		\node [style=none] (55) at (0.75, 8.25) {};
		\node [style=none] (56) at (0.75, 9.25) {};
		\node [style=none] (57) at (1.75, 5) {};
		\node [style=none] (58) at (-1.25, 6.25) {$\alpha$};
		\node [style=none] (59) at (1.5, 4.5) {$X$};
		\node [style=none] (60) at (1.25, 9.75) {$(X\otimes Y)^{\vee \vee}$};
		\node [style=none] (61) at (0.75, 5.5) {};
		\node [style=none] (62) at (1.25, 5.5) {};
		\node [style=none] (65) at (2.5, 5.75) {};
		\node [style=none] (66) at (1.5, 8.25) {};
		\node [style=none] (67) at (1.5, 9.25) {};
		\node [style=none] (68) at (2.5, 5) {};
		\node [style=none] (69) at (2.75, 4.5) {$Y$};
		\node [style=flat box 2] (70) at (1.25, 8.25) {$\omega$};
		\node [style=none] (71) at (-2.25, 7.5) {$=$};
		\node [style=none] (72) at (3.75, 9.25) {};
		\node [style=none] (74) at (1.25, 5.5) {};
	\end{pgfonlayer}
	\begin{pgfonlayer}{edgelayer}
		\draw (6.center) to (5);
		\draw [in=90, out=-90] (5) to (4.center);
		\draw (4.center) to (7.center);
		\draw [style=blarrow] (13.center) to (15.center);
		\draw [style=blue, in=180, out=-90] (0.center) to (13.center);
		\draw [style=over, in=-180, out=0] (15.center) to (1.center);
		\draw (20.center) to (19);
		\draw [in=90, out=-90] (19) to (18.center);
		\draw (18.center) to (21.center);
		\draw [style=blarrow] (25.center) to (26.center);
		\draw [style=blue, in=180, out=0] (16.center) to (25.center);
		\draw [style=over, in=-90, out=0] (26.center) to (17.center);
		\draw (34.center) to (33);
		\draw [in=90, out=-90] (33) to (32.center);
		\draw (32.center) to (35.center);
		\draw [style=blarrow] (39.center) to (40.center);
		\draw [style=blue, in=180, out=-90] (30.center) to (39.center);
		\draw [style=over, in=-180, out=0] (40.center) to (31.center);
		\draw (45.center) to (44);
		\draw [in=90, out=-90] (44) to (43.center);
		\draw (43.center) to (46.center);
		\draw [style=blarrow] (50.center) to (51.center);
		\draw [style=bl, in=180, out=0] (41.center) to (50.center);
		\draw [style=over, in=-90, out=0, looseness=0.75] (51.center) to (42.center);
		\draw (56.center) to (55.center);
		\draw [in=90, out=-90] (55.center) to (54.center);
		\draw (54.center) to (57.center);
		\draw [style=blarrow] (61.center) to (62.center);
		\draw [style=bl, in=180, out=-90] (52.center) to (61.center);
		\draw [style=over, in=-90, out=0] (62.center) to (53.center);
		\draw (67.center) to (66.center);
		\draw [in=90, out=-90] (66.center) to (65.center);
		\draw (65.center) to (68.center);
		\draw [style=over, in=-90, out=0] (74.center) to (72.center);
	\end{pgfonlayer}
\end{tikzpicture} 
\end{align}
Moreover, we can directly read off from~\eqref{eqnpivotalexpl} 
that the component $\xi_I$ is just the 
 coevaluation $I\to \alpha^{-1}\otimes \alpha$. The canonical map $I \to D^2I$ from~\cite[Equation~(2.10)]{bd}
  is 
  the inverse of the isomorphism $D^2I\cong D (I^\vee \otimes \alpha) \cong \alpha^{-1} \otimes \alpha \cong I $, where the last isomorphism is the evaluation. Hence, $\xi_I$ agrees indeed with the canonical map $I\to D^2 I$ from~\cite[Equation~(2.10)]{bd}.

\item
	In order to conclude the proof that $\xi$ is a pivotal structure for the Grothendieck-Verdier duality $D= -^\vee \otimes  \alpha$, it suffices 
	by \cite[Proposition~6.7]{bd}
	to show that the component $\xi_{\alpha} :  \alpha\to D^2 \alpha $ is the canonical isomorphism from the dualizing object to its double Grothendieck-Verdier dual.
	The latter morphism is the canonical isomorphism $\alpha \cong \alpha^{-1} \otimes \alpha \otimes \alpha = D^2 \alpha$ and agrees with the canonical isomorphism $\bar \chi^{D^2}$ (see Lemma~\ref{lemmachi}) for the monoidal autoequivalence $D^2 = \alpha^{-1}\otimes -^{\vee\vee} \otimes \alpha$. 
	Now the desired statement takes the form $\xi_{\alpha}=\bar \chi^{D^2}$ and is exactly the content of Lemma~\ref{lemmachi2}.	
	\end{itemize}
	\end{proof}

By the dual version of~\cite[Proposition~7.10]{bd} there is a canonical monoidal isomorphism \begin{align}\label{eqnvartheta}\vartheta^{\pm} : J^{\pm 1} \to D^2
	\end{align} 
	from the Joyal-Street equivalence $J$ (or its inverse $J^-=J^{-1}$, respectively)
	to the square of any Grothendieck-Verdier duality
	 on a braided monoidal category $\cat{B}$: In our dual conventions, the 
 morphism
$\vartheta^{\pm}_X \colon X \to D^2 X$ is characterized as the unique morphism such that
\begin{align} \cat{B}(K, Y\otimes X ) \ra{{(\id_Y\otimes \vartheta^{\pm}_X )_* }} \cat{B}(K, Y\otimes D^2 X ) \end{align} is the map \begin{align} \label{eqncn} \cat{B}(K, Y\otimes X ) \ra{(c_{X,Y}^{\pm})_* } \cat{B}(K, X \otimes Y) \ra{\nu_{X,Y}} \cat{B}(K, Y\otimes D^2 X ) \ , \end{align} where 
\begin{itemize}
	\item
the first isomorphism is induced by the braiding and its inverse, respectively, \item and
the second
isomorphism is induced by the Grothendieck-Verdier structure, i.e.\
the isomorphisms $\cat{B}(K,X\otimes Y)\cong \cat{B}(DX,Y)$ and the fact the $D$ is an anti-equivalence:
\begin{align}\label{Eq: nu}
	\nu_{X,Y}:
	\cat{B}(K,X\otimes Y) \cong \cat{B}(DX,Y)\stackrel{D}{\cong}
	 \cat{B}(DY,D^2 X)\cong \cat{B}(K,Y\otimes D^2 X) \ . 
	\end{align}
\end{itemize}

\begin{corollary}\label{corbalancing}
	For any pivotal finite tensor category $\cat{C}$, the braided category $Z(\cat{C})$ has a unique balancing $\theta$ with $\vartheta_X^+ \theta_X = \xi_X$ for $X\in Z(\cat{C})$, where $\xi$ is the pivotal structure for the Grothendieck-Verdier duality from Proposition~\ref{proppivotal}.
	\end{corollary}

\begin{proof}
	Clearly, $\theta$ must be defined as the composition
\begin{align}
\theta : \id_{Z(\cat{C})} \ra{\xi} D^2 \ra{(\vartheta^+)^{-1}} J \ , \label{eqnourtheta}
\end{align}
where $\xi$ is from 
Proposition~\ref{proppivotal}
and $\vartheta^+$ from~\eqref{eqnvartheta}.
The isomorphisms $\xi$ and $\vartheta^+$ are monoidal. This is true for
 $\xi$ by Proposition~\ref{proppivotal} and for $\vartheta^+$ by \cite[Proposition~7.10]{bd}.  Therefore, $\theta$ is monoidal.
A monoidal natural isomorphism from the identity to $J$
 is exactly a balancing. This finishes the proof.
\end{proof}

One of the most critical steps towards the main result is the following:

\begin{proposition}\label{propribbon}
	Let $\cat{C}$ be a pivotal finite tensor category.
	Then $\theta_{DX}=D\theta_X$ for $X \in Z(\cat{C})$ and the balancing $\theta$ from Corollary~\ref{corbalancing}. 
	\end{proposition}

For the proof, we will use the following criterion:

\begin{proposition}[$\text{Boyarchenko-Drinfeld \cite[Proposition~8.1 \& Corollary~9.3]{bd}}$]\label{propbdbal}
	Let $\cat{B}$ be a braided 
	Groth\-en\-dieck-Verdier category $\cat{B}$.
	A balancing $\theta: \id_{\cat{B}} \to J$ on $\cat{B}$
	satisfies $\theta_{DX}=D\theta_X$ for all $X\in \cat{B}$ if and only if 
	the pivotal structure
	\begin{align}
		\xi : \id_{\cat{B}} \ra{\theta} J \ra{\vartheta^+} D^2\label{defxi}
	\end{align}
	associated to $\theta$ satisfies $\xi^2 = \gamma : \id_{\cat{B}} \to D^4$, where $\gamma := \vartheta^+ \vartheta^-$ is the canonical isomorphism $\id_{Z(\cat{C})} \to D^4$. 
\end{proposition}

\begin{proof}[\slshape Proof of Proposition~\ref{propribbon}]
By Proposition~\ref{propbdbal} and~\eqref{eqnourtheta}
it now suffices to prove $\xi^2 = \gamma$,
where $\xi$ was introduced in Proposition~\ref{proppivotal}.
We will obtain a proof of the desired equality
$\xi^2 = \gamma$
 by providing a formula for $\gamma:\id_{Z(\cat{C})} \to D^4$ using the graphical calculus of the braided category $Z(\cat{C})$. Deriving the formula is slightly involved, but once we have it,  the equality $\xi^2=\gamma$ can be read off.

 As a first step, we have to spell out the isomorphisms $\vartheta^\pm$ from~\eqref{eqnvartheta}
 for the 
 braided Grothendieck-Verdier category $Z(\cat{C})$:
 For the following computation, it is important to have an explicit description of $D^{-1}$. It
 sends an  object $X$ to $D^{-1}X = \alpha \otimes{}^\vee X$, where ${}^\vee X$ is the
 left dual of $X$. 
  By the Yoneda Lemma $\vartheta^{\pm}$ is the image of the identity under the isomorphism 
\begin{align} Z(\cat{C})(X,X) \cong Z(\cat{C})(\alpha, D^{-1} X \otimes X ) \ra{\eqref{eqncn}} Z(\cat{C})(\alpha, D^{-1} X \otimes D^2 X ) \cong Z(\cat{C})(X,D^2 X) \, , \label{eqncompmaps}
	\end{align} 
\normalsize	
 which we compute using the graphical calculus (we first give the `$+$'-variant):
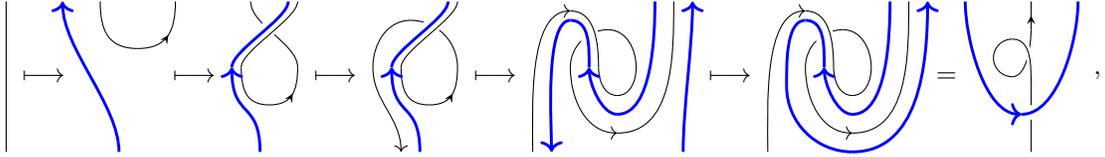
\begin{align}\label{Eq: nu+}
	\begin{tikzpicture}
	\begin{pgfonlayer}{nodelayer}
		\node [style=none] (0) at (-20, 6) {};
		\node [style=none] (1) at (-20, 10) {};
		\node [style=none] (2) at (-19, 8) {};
		\node [style=none] (3) at (-19, 8) {$\longmapsto$};
		\node [style=none] (4) at (-17.5, 10) {};
		\node [style=none] (5) at (-16.5, 8.75) {};
		\node [style=none] (6) at (-15.5, 10) {};
		\node [style=none] (7) at (-17, 6) {};
		\node [style=none] (8) at (-18.5, 10) {};
		\node [style=none] (9) at (-15, 8) {};
		\node [style=none] (10) at (-15, 8) {$\longmapsto$};
		\node [style=none] (11) at (-13.75, 8.25) {};
		\node [style=none] (12) at (-13, 7.25) {};
		\node [style=none] (13) at (-12.25, 8.25) {};
		\node [style=none] (14) at (-13.25, 6) {};
		\node [style=none] (15) at (-14, 8.25) {};
		\node [style=none] (16) at (-13.5, 10) {};
		\node [style=none] (17) at (-12.5, 10) {};
		\node [style=none] (18) at (-12.25, 10) {};
		\node [style=none] (19) at (-11.25, 8) {};
		\node [style=none] (20) at (-11.25, 8) {$\longmapsto$};
		\node [style=none] (21) at (-9.5, 8.25) {};
		\node [style=none] (22) at (-8.75, 7.25) {};
		\node [style=none] (23) at (-8, 8.25) {};
		\node [style=none] (24) at (-9, 6) {};
		\node [style=none] (25) at (-9.75, 8.25) {};
		\node [style=none] (26) at (-9.25, 9.5) {};
		\node [style=none] (27) at (-8.25, 10) {};
		\node [style=none] (28) at (-8, 10) {};
		\node [style=none] (29) at (-10.25, 8.25) {};
		\node [style=none] (30) at (-9.5, 6) {};
		\node [style=none] (31) at (-7, 8) {};
		\node [style=none] (32) at (-7, 8) {$\longmapsto$};
		\node [style=none] (33) at (-4.25, 8.25) {};
		\node [style=none] (34) at (-3.75, 7.5) {};
		\node [style=none] (35) at (-3.25, 8.25) {};
		\node [style=none] (36) at (-3.75, 7) {};
		\node [style=none] (37) at (-4.5, 8.25) {};
		\node [style=none] (38) at (-4, 9.25) {};
		\node [style=none] (39) at (-5, 9.5) {};
		\node [style=none] (40) at (-5, 9.75) {};
		\node [style=none] (41) at (-5, 8.25) {};
		\node [style=none] (42) at (-3.75, 6.5) {};
		\node [style=none] (43) at (-5.5, 6) {};
		\node [style=none] (44) at (-6, 6) {};
		\node [style=none] (45) at (-2.75, 10) {};
		\node [style=none] (46) at (-2.25, 10) {};
		\node [style=none] (48) at (-1.75, 10) {};
		\node [style=none] (49) at (-2, 6) {};
		\node [style=none] (50) at (-0.75, 8) {};
		\node [style=none] (51) at (-0.75, 8) {$\longmapsto$};
		\node [style=none] (52) at (2, 8.25) {};
		\node [style=none] (53) at (2.5, 7.5) {};
		\node [style=none] (54) at (3, 8.25) {};
		\node [style=none] (55) at (2.5, 7) {};
		\node [style=none] (56) at (1.75, 8.25) {};
		\node [style=none] (57) at (2.25, 9.25) {};
		\node [style=none] (58) at (1.25, 9.5) {};
		\node [style=none] (59) at (1.25, 9.75) {};
		\node [style=none] (60) at (1.25, 8.25) {};
		\node [style=none] (61) at (2.5, 6.5) {};
		\node [style=none] (62) at (2.5, 6) {};
		\node [style=none] (63) at (0.25, 6) {};
		\node [style=none] (64) at (3.5, 10) {};
		\node [style=none] (65) at (4, 10) {};
		\node [style=none] (66) at (4.5, 10) {};
		\node [style=none] (67) at (2.5, 6) {};
		\node [style=none] (68) at (0.75, 7.75) {};
		\node [style=none] (69) at (1, 8) {};
		\node [style=none] (70) at (5, 8) {};
		\node [style=none] (71) at (5, 8) {$=$};
		\node [style=none] (72) at (7.25, 6) {};
		\node [style=none] (73) at (7.25, 7.75) {};
		\node [style=none] (74) at (6.75, 9) {};
		\node [style=none] (75) at (6.25, 8.5) {};
		\node [style=none] (76) at (6.75, 8) {};
		\node [style=none] (77) at (7.25, 9.25) {};
		\node [style=none] (78) at (7.25, 10) {};
		\node [style=none] (79) at (5.5, 10) {};
		\node [style=none] (80) at (8.5, 10) {};
		\node [style=none] (81) at (7, 7) {};
		\node [style=none] (82) at (5.5, 10) {};
	\end{pgfonlayer}
	\begin{pgfonlayer}{edgelayer}
		\draw (0.center) to (1.center);
		\draw [in=180, out=-90, looseness=1.25] (4.center) to (5.center);
		\draw [style=mid arrow, in=-90, out=0, looseness=1.25] (5.center) to (6.center);
		\draw [style=blarrow, in=-90, out=90, looseness=0.75] (7.center) to (8.center);
		\draw [in=180, out=-90, looseness=1.25] (11.center) to (12.center);
		\draw [style=mid arrow, in=-90, out=0, looseness=1.25] (12.center) to (13.center);
		\draw [style=blarrow, in=-90, out=90, looseness=1.50] (14.center) to (15.center);
		\draw [in=-90, out=90, looseness=0.75] (13.center) to (16.center);
		\draw [style=over, in=-90, out=90, looseness=0.50] (15.center) to (17.center);
		\draw [style=black over, in=-90, out=90, looseness=0.50] (11.center) to (18.center);
		\draw [in=180, out=-90, looseness=1.25] (21.center) to (22.center);
		\draw [style=mid arrow, in=-90, out=0, looseness=1.25] (22.center) to (23.center);
		\draw [style=blarrow, in=-90, out=90, looseness=1.50] (24.center) to (25.center);
		\draw [in=0, out=90] (23.center) to (26.center);
		\draw [style=over, in=-90, out=90, looseness=0.50] (25.center) to (27.center);
		\draw [style=black over, in=-90, out=90, looseness=0.50] (21.center) to (28.center);
		\draw [style=black over, in=90, out=180] (26.center) to (29.center);
		\draw [style=end arrow, in=90, out=-90] (29.center) to (30.center);
		\draw [in=180, out=-90] (33.center) to (34.center);
		\draw [in=-90, out=0] (34.center) to (35.center);
		\draw [style=blarrow, in=-90, out=-180, looseness=0.75] (36.center) to (37.center);
		\draw [in=0, out=90] (35.center) to (38.center);
		\draw [in=90, out=180] (38.center) to (41.center);
		\draw [style=end arrow, in=-180, out=-90] (41.center) to (42.center);
		\draw [style=blarrow, in=90, out=-180, looseness=0.50] (39.center) to (43.center);
		\draw [style=end arrow, in=180, out=90, looseness=0.75] (44.center) to (40.center);
		\draw [style=blue, in=0, out=-90, looseness=0.75] (45.center) to (36.center);
		\draw [in=-90, out=0] (42.center) to (46.center);
		\draw [style=blarrow, in=270, out=90, looseness=0.75] (49.center) to (48.center);
		\draw [in=180, out=-90] (52.center) to (53.center);
		\draw [in=-90, out=0] (53.center) to (54.center);
		\draw [style=blarrow, in=-90, out=-180, looseness=0.75] (55.center) to (56.center);
		\draw [in=0, out=90] (54.center) to (57.center);
		\draw [style=black over, in=90, out=180] (57.center) to (60.center);
		\draw [style=end arrow, in=-180, out=-90] (60.center) to (61.center);
		\draw [style=end arrow, in=180, out=90, looseness=0.75] (63.center) to (59.center);
		\draw [style=blue, in=0, out=-90, looseness=0.75] (64.center) to (55.center);
		\draw [in=-90, out=0] (61.center) to (65.center);
		\draw [style=blarrow, in=270, out=0] (67.center) to (66.center);
		\draw [style=blue, in=90, out=180, looseness=0.75] (58.center) to (68.center);
		\draw [style=blue, in=180, out=-90] (68.center) to (67.center);
		\draw [in=90, out=-180] (74.center) to (75.center);
		\draw [in=-180, out=-90] (75.center) to (76.center);
		\draw [in=-90, out=0, looseness=0.75] (76.center) to (77.center);
		\draw [style=mid arrow] (77.center) to (78.center);
		\draw (72.center) to (73.center);
		\draw [style=black over, in=0, out=90] (73.center) to (74.center);
		\draw [style=blarrow, in=180, out=-90, looseness=0.75] (82.center) to (81.center);
		\draw [style=over, in=-90, out=0, looseness=0.75] (81.center) to (80.center);
		\draw [style=black over, in=0, out=90] (33.center) to (40.center);
		\draw [style=over, in=0, out=90] (37.center) to (39.center);
		\draw [style=over, in=0, out=90] (56.center) to (58.center);
		\draw [style=black over, in=0, out=90] (52.center) to (59.center);
	\end{pgfonlayer}
\end{tikzpicture} \ , 
\end{align}
where \begin{itemize}
	\item the first map is the identification $Z(\cat{C})(X,X) \cong Z(\cat{C})(\alpha, D^{-1} X \otimes X )$, 
	\item the second is the braiding part of~\eqref{eqncn},\item the next two maps are the first
two identifications 
from~\eqref{Eq: nu} for the definition of $\nu$, 
\item and the last map is the composition of the identification 
\begin{align}
Z(\cat{C})(D(D^{-1}X),D^2 X)\cong Z(\cat{C})(\alpha , D^{-1} X \otimes D^2 X)\end{align}
from~\eqref{Eq: nu} with the identification $Z(\cat{C})(\alpha, D^{-1} X \otimes D^2 X ) \cong Z(\cat{C})(X,D^2 X)$ from~\eqref{eqncompmaps}.\end{itemize}
In summary, we find
\begin{align}\label{Eq: nu+ alone}
	\begin{tikzpicture}
	\begin{pgfonlayer}{nodelayer}
		\node [style=none] (0) at (5, 8) {};
		\node [style=none] (1) at (5, 8) {$=$};
		\node [style=none] (2) at (7.25, 6) {};
		\node [style=none] (3) at (7.25, 7.75) {};
		\node [style=none] (4) at (6.75, 9) {};
		\node [style=none] (5) at (6.25, 8.5) {};
		\node [style=none] (6) at (6.75, 8) {};
		\node [style=none] (7) at (7.25, 9.25) {};
		\node [style=none] (8) at (7.25, 10) {};
		\node [style=none] (9) at (5.5, 10) {};
		\node [style=none] (10) at (8.5, 10) {};
		\node [style=none] (11) at (7, 7) {};
		\node [style=none] (12) at (5.5, 10) {};
		\node [style=none] (13) at (4, 8) {$\vartheta^+$};
	\end{pgfonlayer}
	\begin{pgfonlayer}{edgelayer}
		\draw [in=90, out=-180] (4.center) to (5.center);
		\draw [in=-180, out=-90] (5.center) to (6.center);
		\draw [in=-90, out=0, looseness=0.75] (6.center) to (7.center);
		\draw [style=mid arrow] (7.center) to (8.center);
		\draw (2.center) to (3.center);
		\draw [style=black over, in=0, out=90] (3.center) to (4.center);
		\draw [style=blarrow, in=180, out=-90, looseness=0.75] (12.center) to (11.center);
		\draw [style=over, in=-90, out=0, looseness=0.75] (11.center) to (10.center);
	\end{pgfonlayer}
\end{tikzpicture} \ .
\end{align}
For the `$-$'-variant (i.e.\ for $\vartheta^-$),
we obtain  the same picture with all overcrossings 
replaced by undercrossings. 

The formula for $\vartheta^\pm$ provides us with a formula for the natural isomorphism $\gamma \colon \id_{Z(\cat{C})} = J^+ \circ J^- \xrightarrow{\vartheta^+ \circ \vartheta^-} D^4$: The component of $\gamma$ at an object $X \in Z(\cat{C})$
is the composition
$\gamma_X \colon X \xrightarrow{\vartheta^+} D^2X \xrightarrow{D^2 \vartheta^- } D^4 X$. Using that the value of $D^2$ on a morphism can be expressed through the pivotal structure via
\begin{align}
\begin{tikzpicture}
	\begin{pgfonlayer}{nodelayer}
		\node [style=none] (0) at (-19.25, 4.25) {};
		\node [style=none] (1) at (-19.25, 10) {};
		\node [style=none] (2) at (-20.5, 7) {$D^2$};
		\node [style=square] (4) at (-19.25, 7) {$f$};
		\node [style=none] (7) at (-15, 7) {};
		\node [style=none] (8) at (-15, 7.5) {};
		\node [style=square] (11) at (-15, 7.25) {$f$};
		\node [style=none] (12) at (-17.5, 10.25) {};
		\node [style=none] (13) at (-17.5, 4.25) {};
		\node [style=none] (14) at (-12.5, 4.25) {};
		\node [style=none] (15) at (-12.5, 10.25) {};
		\node [style=none] (16) at (-17.5, 7.5) {};
		\node [style=none] (17) at (-17.5, 7) {};
		\node [style=none] (18) at (-12.5, 7.5) {};
		\node [style=none] (19) at (-12.5, 7) {};
		\node [style=none] (20) at (-17.5, 7) {};
		\node [style=none] (21) at (-16.5, 4.25) {};
		\node [style=none] (22) at (-15.5, 9.25) {};
		\node [style=none] (23) at (-14, 7.25) {};
		\node [style=none] (24) at (-14.5, 6.25) {};
		\node [style=none] (25) at (-15.5, 8.25) {};
		\node [style=none] (26) at (-16, 7.25) {};
		\node [style=none] (27) at (-14.5, 5.25) {};
		\node [style=none] (29) at (-13.25, 10.25) {};
		\node [style=none] (30) at (-18.5, 7) {$=$};
		\node [style=none] (34) at (-11, 10.25) {};
		\node [style=none] (35) at (-11, 4.25) {};
		\node [style=none] (36) at (-8, 4.25) {};
		\node [style=none] (37) at (-8, 10.25) {};
		\node [style=none] (38) at (-11, 7.5) {};
		\node [style=none] (39) at (-11, 7) {};
		\node [style=none] (40) at (-8, 7.5) {};
		\node [style=none] (41) at (-8, 7) {};
		\node [style=none] (42) at (-11, 7) {};
		\node [style=none] (52) at (-11.75, 7.25) {$=$};
		\node [style=none] (53) at (-9.5, 4.25) {};
		\node [style=none] (54) at (-9.5, 10.25) {};
		\node [style=square] (55) at (-9.5, 7.25) {$f$};
		\node [style=square] (56) at (-9.5, 9.25) {$\omega$};
		\node [style=square] (57) at (-9.5, 5.25) {$\omega^{-1}$};
		\node [style=none] (59) at (-16.25, 3.75) {$X^{\vee \vee}$};
		\node [style=none] (60) at (-13.5, 10.75) {$Y^{\vee \vee}$};
		\node [style=none] (61) at (-9.5, 3.75) {$X^{\vee \vee}$};
		\node [style=none] (62) at (-9.5, 10.75) {$Y^{\vee \vee}$};
		\node [style=none] (63) at (-17.5, 3.75) {};
		\node [style=none] (64) at (-12.5, 3.75) {};
		\node [style=none] (65) at (-12.5, 3.75) {$\alpha$};
		\node [style=none] (66) at (-17.5, 3.75) {};
		\node [style=none] (67) at (-17.5, 3.75) {$\alpha^{\vee}$};
		\node [style=none] (68) at (-11, 3.75) {$\alpha^{\vee}$};
		\node [style=none] (69) at (-8, 3.75) {$\alpha$};
		\node [style=none] (70) at (-19.25, 3.75) {};
		\node [style=none] (71) at (-19.25, 3.75) {$X$};
		\node [style=none] (72) at (-19.25, 10.5) {$Y$};
	\end{pgfonlayer}
	\begin{pgfonlayer}{edgelayer}
		\draw (0.center) to (1.center);
		\draw [style=over] (7.center) to (8.center);
		\draw [style=blarrow] (16.center) to (20.center);
		\draw [style=blarrow] (19.center) to (18.center);
		\draw [style=blue] (14.center) to (19.center);
		\draw [style=blue] (18.center) to (15.center);
		\draw [style=blue] (20.center) to (13.center);
		\draw [style=blue] (16.center) to (12.center);
		\draw [style=mid arrow, in=-180, out=90, looseness=0.75] (21.center) to (22.center);
		\draw [in=90, out=0] (22.center) to (23.center);
		\draw [in=0, out=-90, looseness=0.75] (23.center) to (24.center);
		\draw [in=-90, out=-180, looseness=0.75] (24.center) to (7.center);
		\draw [in=0, out=90] (8.center) to (25.center);
		\draw [in=90, out=-180] (25.center) to (26.center);
		\draw [in=-180, out=-90] (26.center) to (27.center);
		\draw [style=blarrow] (38.center) to (42.center);
		\draw [style=blarrow] (41.center) to (40.center);
		\draw [style=blue] (36.center) to (41.center);
		\draw [style=blue] (40.center) to (37.center);
		\draw [style=blue] (42.center) to (35.center);
		\draw [style=blue] (38.center) to (34.center);
		\draw (53.center) to (54.center);
		\draw [style=mid arrow, in=-90, out=0, looseness=0.75] (27.center) to (29.center);
	\end{pgfonlayer}
\end{tikzpicture} \qquad , 
\end{align}
we find: 
\begin{align}\label{Eq: gamma}
\begin{tikzpicture}
	\begin{pgfonlayer}{nodelayer}
		\node [style=none] (72) at (-6.75, 4.5) {};
		\node [style=none] (73) at (-6.75, 6) {};
		\node [style=none] (74) at (-7.25, 7.25) {};
		\node [style=none] (75) at (-7.75, 6.75) {};
		\node [style=none] (76) at (-7.25, 6.25) {};
		\node [style=none] (77) at (-6.75, 7.5) {};
		\node [style=none] (78) at (-6.75, 9.5) {};
		\node [style=none] (80) at (-4.25, 12.75) {};
		\node [style=none] (81) at (-7, 5.5) {};
		\node [style=none] (82) at (-9.5, 12.75) {};
		\node [style=none] (83) at (-6.75, 8.75) {};
		\node [style=none] (84) at (-6.75, 10) {};
		\node [style=none] (85) at (-7.25, 11.25) {};
		\node [style=none] (86) at (-7.75, 10.75) {};
		\node [style=none] (87) at (-7.25, 10.25) {};
		\node [style=none] (88) at (-6.75, 11.5) {};
		\node [style=none] (89) at (-6.75, 12.75) {};
		\node [style=none] (90) at (-8.5, 12.25) {};
		\node [style=none] (91) at (-5.5, 12.75) {};
		\node [style=none] (92) at (-7, 9.25) {};
		\node [style=none] (93) at (-8.5, 12.75) {};
		\node [style=square] (95) at (-6.75, 8.25) {$\omega^{-1}$};
		\node [style=square] (96) at (-6.75, 12) {$\omega$};
		\node [style=square] (97) at (-8.5, 12) {$\omega$};
		\node [style=square] (98) at (-5.5, 12) {$\omega$};
		\node [style=none] (99) at (0, 4.5) {};
		\node [style=none] (100) at (0, 8.5) {};
		\node [style=none] (101) at (-0.5, 9.75) {};
		\node [style=none] (102) at (-0.75, 9.25) {};
		\node [style=none] (103) at (-0.5, 9) {};
		\node [style=none] (104) at (0, 10) {};
		\node [style=none] (106) at (2.25, 12.75) {};
		\node [style=none] (107) at (-0.25, 6.5) {};
		\node [style=none] (108) at (-2.75, 12.75) {};
		\node [style=none] (110) at (0, 10) {};
		\node [style=none] (111) at (-0.5, 11.25) {};
		\node [style=none] (112) at (-0.75, 10.75) {};
		\node [style=none] (113) at (-0.5, 10.5) {};
		\node [style=none] (114) at (0, 11.5) {};
		\node [style=none] (115) at (0, 12.75) {};
		\node [style=none] (116) at (-1.75, 12.25) {};
		\node [style=none] (117) at (1.25, 12.75) {};
		\node [style=none] (118) at (-0.25, 7.75) {};
		\node [style=none] (119) at (-1.75, 12.75) {};
		\node [style=square] (122) at (-1.75, 12) {$\omega$};
		\node [style=square] (123) at (1.25, 12) {$\omega$};
		\node [style=none] (124) at (-3.5, 9) {$=$};
		\node [style=none] (125) at (0, 7.25) {};
		\node [style=none] (127) at (-11.25, 9) {$\gamma \ \ =$};
		\node [style=none] (128) at (-1, 11.5) {};
		\node [style=none] (129) at (0.5, 11.5) {};
		\node [style=none] (130) at (-1, 8.75) {};
		\node [style=none] (131) at (0.5, 8.75) {};
	\end{pgfonlayer}
	\begin{pgfonlayer}{edgelayer}
		\draw [in=90, out=-180] (74.center) to (75.center);
		\draw [in=-180, out=-90] (75.center) to (76.center);
		\draw [in=-90, out=0, looseness=0.75] (76.center) to (77.center);
		\draw (72.center) to (73.center);
		\draw [style=black over, in=0, out=90] (73.center) to (74.center);
		\draw [style=blarrow, in=180, out=-90, looseness=0.75] (82.center) to (81.center);
		\draw [style=over, in=-90, out=0, looseness=0.75] (81.center) to (80.center);
		\draw [in=90, out=-180] (85.center) to (86.center);
		\draw [in=-180, out=-90] (86.center) to (87.center);
		\draw [style=mid arrow] (88.center) to (89.center);
		\draw [in=0, out=90] (84.center) to (85.center);
		\draw [style=blarrow, in=180, out=-90, looseness=0.75] (93.center) to (92.center);
		\draw [style=blue, in=-90, out=0, looseness=0.75] (92.center) to (91.center);
		\draw [in=90, out=-180] (101.center) to (102.center);
		\draw [in=-180, out=-90] (102.center) to (103.center);
		\draw [in=-90, out=0, looseness=0.75] (103.center) to (104.center);
		\draw [style=black over, in=0, out=90] (100.center) to (101.center);
		\draw [style=blarrow, in=180, out=-90, looseness=0.75] (108.center) to (107.center);
		\draw [in=90, out=-180] (111.center) to (112.center);
		\draw [in=-180, out=-90] (112.center) to (113.center);
		\draw [style=mid arrow] (114.center) to (115.center);
		\draw [in=0, out=90] (110.center) to (111.center);
		\draw [style=blarrow, in=180, out=-90, looseness=0.75] (119.center) to (118.center);
		\draw [style=over, in=-90, out=0, looseness=0.75] (118.center) to (117.center);
		\draw (99.center) to (125.center);
		\draw [style=black over] (125.center) to (100.center);
		\draw [style=over, in=-90, out=0, looseness=0.75] (107.center) to (106.center);
		\draw [style=black over, in=-90, out=0, looseness=0.75] (113.center) to (114.center);
		\draw (77.center) to (83.center);
		\draw [style=black over] (83.center) to (84.center);
		\draw [style=black over, in=-90, out=0, looseness=0.75] (87.center) to (88.center);
		\draw [style=rdots] (130.center) to (131.center);
		\draw [style=rdots] (131.center) to (129.center);
		\draw [style=rdots] (129.center) to (128.center);
		\draw [style=rdots] (128.center) to (130.center);
	\end{pgfonlayer}
\end{tikzpicture} \ , 
\end{align}
where we have drawn a dotted box to single out an isomorphism 
$\Delta_X:X \to X^{\vee\vee\vee\vee}$ in $Z(\cat{C})$ that we will investigate now: 
This isomorphism is the Radford isomorphism for $Z(\cat{C})$ as follows from \cite[Theorem~8.10.7]{egno} (see also \cite[Section~5.5]{shimizuribbon}) when combined with the fact that the distinguished invertible object of $Z(\cat{C})$ is the monoidal unit~\cite[Proposition~4.4 \& 4.5]{eno-d}. 
It is shown in the proof of \cite[Theorem~5.8]{shimizuribbon}
that $\Delta_X \otimes \alpha : X \otimes \alpha \to X ^{\vee\vee\vee\vee}\otimes \alpha$ agrees with
\begin{align}
	X\otimes \alpha \ra{\sigma_X} \alpha \otimes X \ra{\delta^+} X^{\vee\vee\vee\vee}\otimes \alpha \ ; \label{Xalpha2}
\end{align}
in other words, this is a relation between the Radford isomorphisms in $\cat{C}$ and $Z(\cat{C})$.  
We can rewrite~\eqref{Xalpha2}
using the half braiding of $\alpha$ from Lemma~\ref{defhalfbraiding} as
\begin{align}
X\otimes \alpha \ra{\sigma_X} \alpha \otimes X \ra{\sigma_\alpha} X\otimes \alpha \ra{\omega^2 } X^{\vee\vee\vee\vee}\otimes \alpha \ \ .
\end{align} 
Inserting this into the formula for $\gamma_X$ from~\eqref{Eq: gamma} we find:
\begin{align}
\begin{tikzpicture}
	\begin{pgfonlayer}{nodelayer}
		\node [style=none] (99) at (-4, 4) {};
		\node [style=none] (100) at (-4, 8) {};
		\node [style=none] (101) at (-4.5, 9.25) {};
		\node [style=none] (102) at (-5, 8.75) {};
		\node [style=none] (103) at (-4.5, 8.25) {};
		\node [style=none] (104) at (-4, 9.5) {};
		\node [style=none] (106) at (-1.75, 12.25) {};
		\node [style=none] (107) at (-4.25, 6) {};
		\node [style=none] (108) at (-6.75, 12.25) {};
		\node [style=none] (110) at (-4, 9.5) {};
		\node [style=none] (111) at (-4.5, 10.75) {};
		\node [style=none] (112) at (-5, 10.25) {};
		\node [style=none] (113) at (-4.5, 9.75) {};
		\node [style=none] (114) at (-4, 11) {};
		\node [style=none] (115) at (-4, 12.25) {};
		\node [style=none] (116) at (-5.75, 11.75) {};
		\node [style=none] (117) at (-2.75, 12.25) {};
		\node [style=none] (118) at (-4.25, 7.25) {};
		\node [style=none] (119) at (-5.75, 12.25) {};
		\node [style=square] (122) at (-5.75, 11.5) {$\omega$};
		\node [style=square] (123) at (-2.75, 11.5) {$\omega$};
		\node [style=none] (125) at (-4, 6.75) {};
		\node [style=none] (126) at (3, 4) {};
		\node [style=none] (127) at (3, 7.75) {};
		\node [style=none] (132) at (6, 12.25) {};
		\node [style=none] (133) at (2.75, 6) {};
		\node [style=none] (134) at (0, 12.25) {};
		\node [style=none] (139) at (2.75, 11) {};
		\node [style=none] (140) at (2.75, 12.25) {};
		\node [style=none] (141) at (1.25, 11.75) {};
		\node [style=none] (142) at (4.25, 12.25) {};
		\node [style=none] (143) at (2.75, 7.25) {};
		\node [style=none] (144) at (1.25, 12.25) {};
		\node [style=square] (145) at (1.25, 11.5) {$\omega$};
		\node [style=square] (146) at (4.25, 11.5) {$\omega$};
		\node [style=none] (147) at (3.25, 6.75) {};
		\node [style=none] (148) at (-0.75, 9) {};
		\node [style=none] (149) at (-0.75, 9) {$=$};
		\node [style=none] (151) at (4.75, 8.75) {};
		\node [style=none] (152) at (4, 9) {};
		\node [style=square] (153) at (2.75, 11.5) {$\omega^2$};
		\node [style=none] (154) at (12, 4) {};
		\node [style=none] (156) at (13.5, 12.25) {};
		\node [style=none] (157) at (10.5, 6) {};
		\node [style=none] (158) at (7.75, 12.25) {};
		\node [style=none] (160) at (10.5, 12.25) {};
		\node [style=none] (161) at (9, 11.75) {};
		\node [style=none] (162) at (12, 12.25) {};
		\node [style=none] (163) at (10.5, 7.75) {};
		\node [style=none] (164) at (9, 12.25) {};
		\node [style=square] (165) at (9, 11.5) {$\omega$};
		\node [style=square] (166) at (12, 11.5) {$\omega$};
		\node [style=none] (168) at (7, 9) {};
		\node [style=none] (169) at (7, 9) {$=$};
		\node [style=none] (171) at (11.75, 9) {};
		\node [style=square] (172) at (10.5, 11.5) {$\omega^2$};
		\node [style=none] (173) at (-8, 9) {};
		\node [style=none] (174) at (-8, 9) {$\gamma \ \ =$};
	\end{pgfonlayer}
	\begin{pgfonlayer}{edgelayer}
		\draw [in=90, out=-180] (101.center) to (102.center);
		\draw [in=-180, out=-90] (102.center) to (103.center);
		\draw [in=-90, out=0, looseness=0.75] (103.center) to (104.center);
		\draw [style=black over, in=0, out=90] (100.center) to (101.center);
		\draw [style=blarrow, in=180, out=-90, looseness=0.75] (108.center) to (107.center);
		\draw [in=90, out=-180] (111.center) to (112.center);
		\draw [in=-180, out=-90] (112.center) to (113.center);
		\draw [style=mid arrow] (114.center) to (115.center);
		\draw [in=0, out=90] (110.center) to (111.center);
		\draw [style=blarrow, in=180, out=-90, looseness=0.75] (119.center) to (118.center);
		\draw [style=over, in=-90, out=0, looseness=0.75] (118.center) to (117.center);
		\draw (99.center) to (125.center);
		\draw [style=black over] (125.center) to (100.center);
		\draw [style=over, in=-90, out=0, looseness=0.75] (107.center) to (106.center);
		\draw [style=black over, in=-90, out=0, looseness=0.75] (113.center) to (114.center);
		\draw [style=blarrow, in=180, out=-90, looseness=0.75] (134.center) to (133.center);
		\draw [style=mid arrow] (139.center) to (140.center);
		\draw [style=blarrow, in=180, out=-90, looseness=0.75] (144.center) to (143.center);
		\draw [in=-90, out=90] (126.center) to (147.center);
		\draw [style=over, in=-90, out=0, looseness=0.75] (133.center) to (132.center);
		\draw [in=-90, out=90, looseness=0.75] (151.center) to (139.center);
		\draw [style=blue, in=-105, out=0, looseness=1.25] (143.center) to (152.center);
		\draw [style=over, in=-90, out=75] (152.center) to (142.center);
		\draw [style=black over, in=-90, out=90] (147.center) to (127.center);
		\draw [style=black over, in=-90, out=90, looseness=0.75] (127.center) to (151.center);
		\draw [style=blarrow, in=180, out=-90, looseness=0.75] (158.center) to (157.center);
		\draw [style=blarrow, in=180, out=-90, looseness=0.75] (164.center) to (163.center);
		\draw [in=-90, out=90] (154.center) to (160.center);
		\draw [style=over, in=-90, out=0, looseness=0.75] (163.center) to (162.center);
		\draw [style=over, in=-90, out=0, looseness=0.75] (157.center) to (156.center);
	\end{pgfonlayer}
\end{tikzpicture}
\end{align}
This agrees with $\xi^2$, as we can see with the formula given for $\xi$ in Remark~\ref{remformulaxi}. This finishes the proof. 
	\end{proof}

If a finite tensor category $\cat{C}$ is pivotal, then $Z(\cat{C})$ inherits a pivotal structure from its rigid duality.
Since $Z(\cat{C})$ is also braided, it comes also with a balancing:
\begin{align}
	\begin{tikzpicture}
	\begin{pgfonlayer}{nodelayer}
		\node [style=none] (72) at (-1, 6) {};
		\node [style=none] (73) at (-1, 7.75) {};
		\node [style=none] (77) at (-1, 9.25) {};
		\node [style=none] (78) at (-1, 10) {};
		\node [style=none] (79) at (-0.25, 9) {};
		\node [style=none] (80) at (-0.25, 8) {};
		\node [style=none] (81) at (0.25, 8.5) {};
		\node [style=none] (82) at (0.25, 8.5) {};
		\node [style=square] (84) at (-1, 6.75) {$\omega$};
		\node [style=none] (85) at (-3, 8) {};
		\node [style=none] (86) at (-2.5, 8) {$\theta^\circ$     $=$};
		\node [style=none] (89) at (-4.5, 7.25) {};
	\end{pgfonlayer}
	\begin{pgfonlayer}{edgelayer}
		\draw [style=mid arrow] (77.center) to (78.center);
		\draw (72.center) to (73.center);
		\draw [in=90, out=0] (79.center) to (82.center);
		\draw [in=0, out=-90, looseness=1.25] (82.center) to (80.center);
		\draw [in=-90, out=180, looseness=0.75] (80.center) to (77.center);
		\draw [style=black over, in=-180, out=90, looseness=0.75] (73.center) to (79.center);
	\end{pgfonlayer}
\end{tikzpicture} \qquad . 
	\label{eqnpivrigid}
\end{align}
We refer to this as the \emph{canonical balanced braided structure} of the Drinfeld center of a pivotal finite tensor category.

The description of $\vartheta^+$ in the proof
of Proposition~\ref{propribbon} can be used to prove the following:

\begin{corollary}\label{corcanbalbr}
	The balancing braided structure on the Drinfeld center $Z(\cat{C})$ of a pivotal finite tensor category $\cat{C}$
	 from Corollary~\ref{corbalancing} 
	agrees with the canonical one.
\end{corollary}

\begin{proof}
	Clearly, the underlying braided monoidal structures agree by definition.
	The statement is only non-trivial for the balancing:
The balancing 
from Corollary~\ref{corbalancing}
is by definition the composition of the pivotal structure $\xi$ from Proposition~\ref{proppivotal}
with the inverse of $\vartheta^+$. 
Therefore, it suffices to prove $\xi = \vartheta^+ \theta^\circ$ for the canonical balancing. 
Indeed, with the expression for $\vartheta^+$ found in~\eqref{Eq: nu+ alone},
 this can be directly verified in the graphical calculus.
\end{proof}

We may now state and prove the main result:

\begin{theorem}\label{thmribbon}
Let $\cat{C}$ be a pivotal finite tensor category.
Then  the distinguished invertible object of $\cat{C}$ equipped with the half braiding induced by the Radford isomorphism and the pivotal structure of $\cat{C}$ is a dualizing object that makes $Z(\cat{C})$ a ribbon Grothendieck-Verdier category.
Up to equivalence, this is the only ribbon Grothendieck-Verdier structure on $Z(\cat{C})$ that extends the canonical balanced braided structure.
\end{theorem}

\begin{proof}
	We just have to piece the results together: Proposition~\ref{proppivotal} and~\ref{propribbon} tell us that the distinguished invertible object of $\cat{C}$ with its half braiding from Lemma~\ref{lemmahalfbraiding} is a dualizing object for $Z(\cat{C})$ whose Grothendieck-Verdier duality is in fact ribbon. 
	The underlying balanced braided structure is the canonical one by Corollary~\ref{corcanbalbr}.
	
	The  proof of the uniqueness statement
	requires entirely different methods. 
	We treat the uniqueness question in Section~\ref{seccyclicstructures}, where we will formally finish this proof on page~\pageref{endofproof}.
	\end{proof}

\begin{example}\label{exgraded2}
	The ribbon Grothendieck-Verdier structure on the Drinfeld center $Z(\Vect_G^{\lambda,d})$ 
	of the pivotal finite tensor category $\Vect_G^{\lambda,d}$ from Example~\ref{exgraded}
	has the following explicit description: The dualizing object is $K=(I, (\omega^{(d)})^2)$ and the duality functor maps an object 
	$(V=\bigoplus_{g\in G}V_g ,\sigma_V)$ to $(V^*, {\omega^{(d)}}^2 \sigma^{*}_V)$ where $V^*$ is the dual object of $V$ in $\Vect_G^{\lambda,d}$, $\sigma^*_V$ is the canonical half braiding on $V^*$~\cite[Section~5.2.1]{TV}, and ${\omega^{(d)}}^2$ is the square of the pivotal structure
	associated to $d:G\to k^\times$. 
\end{example}

\begin{corollary}\label{corspherical}
	For the 
ribbon Grothendieck-Verdier duality on the Drinfeld center $Z(\cat{C})$ of a pivotal finite tensor category $\cat{C}$ described in Theorem~\ref{thmribbon}, the following are equivalent:
\begin{pnum}
	\item The dualizing object is isomorphic to the unit in $Z(\cat{C})$.\label{corsphericali}
	\item $\cat{C}$ is spherical.\label{corsphericalii}
	\item The Grothendieck-Verdier duality of $Z(\cat{C})$ agrees with the rigid duality.\label{corsphericaliii}
	\item $Z(\cat{C})$ with its canonical balanced braided structure is a modular category.\label{corsphericaliv}
\end{pnum} 
	\end{corollary} 

\begin{proof}
	The dualizing object in $Z(\cat{C})$
	(by its definition in Proposition~\ref{proppivotal})
	 is $\alpha$  with the half braiding from Lemma~\ref{defhalfbraiding}. With this knowledge, \ref{corsphericali} $\Leftrightarrow$ \ref{corsphericalii}
	is already a part of Lemma~\ref{defhalfbraiding}.
	
	Since the Grothendieck-Verdier duality is the rigid duality tensored with the dualizing object, 
	\ref{corsphericali} $\Leftrightarrow$ \ref{corsphericaliii} is clear. 
	
	Finally, if~\ref{corsphericaliii} holds, then $Z(\cat{C})$ is ribbon (in the traditional sense) and 
	factorizable by \cite[Proposition~4.4]{eno-d}
	and has therefore a trivial Müger center by \cite[Theorem~1.1]{shimizumodular}. 
	This proves \ref{corsphericaliii} $\Rightarrow$ \ref{corsphericaliv}. 
	For the converse \ref{corsphericaliv} $\Rightarrow$ \ref{corsphericaliii}
	note that if the balanced braided
	category $Z(\cat{C})$ is modular, then it must be rigid and ribbon.
	This makes $Z(\cat{C})$ into a ribbon Grothendieck-Verdier category
	$Z(\cat{C})^\text{rigid}$ with rigid duality.
	In order to conclude that \ref{corsphericaliii} holds, i.e.\ that the Grothendieck-Verdier duality of $Z(\cat{C})$, equipped with the structure from Theorem~\ref{thmribbon}, is rigid, it suffices to prove
	$Z(\cat{C})=Z(\cat{C})^\text{rigid}$ as ribbon Grothendieck-Verdier categories. 
	This is indeed a priori not clear, but follows from the uniqueness statement contained in Theorem~\ref{thmribbon} because the underlying balanced braided structure of both categories agree.
	\end{proof}

\section{The low-dimensional topology implications\label{secldt}}
Ribbon Grothendieck-Verdier categories are intimately related to low-dimensional topology.
More precisely, they are equivalent to cyclic algebras over the framed little disks operad~\cite[Theorem 5.13]{cyclic}. 
Here the algebras take values in a symmetric monoidal bicategory $\Lexf$ of linear categories (subject to finiteness conditions), left exact functors and natural transformations.
Furthermore, it is shown in~\cite{mwansular} that ribbon Grothendieck-Verdier categories are equivalent to
so-called \emph{ansular functors}, i.e.\  systems of handlebody group representations that are compatible with gluing. In particular, 
an ansular functor assigns to a three-dimensional handlebody, possibly with a fixed number of disks embedded in its boundary, a vector space together with an action of the 
handlebody group of that handlebody.
From this perspective, Theorem~\ref{thmribbon} tells us the following:
For any pivotal finite tensor category $\cat{C}$,
we obtain an ansular functor, namely the one built from the ribbon Grothendieck-Verdier category $Z(\cat{C})$.
In fact, the factorization homology methods developed in~\cite{brochierwoike} will prove that $Z(\cat{C})$ yields in fact a \emph{modular functor}, i.e.\ we get representations of mapping class groups of surfaces
 and not just handlebody groups. 
However, this is not needed for the considerations in this section.
We will only be interested in the vector spaces
$B^{Z(\cat{C})}_g$ that the ansular functor for $Z(\cat{C})$ assigns to a handlebody whose boundary is a surface of genus $g$. These vector spaces are often referred to as the \emph{spaces of conformal blocks} of $Z(\cat{C})$, and they can be explicitly described as we will see below. 
We will for simplicity be interested only in closed surfaces. 
For example, 
by \cite[Corollary~6.3]{mwansular} the vector space $B_0^{Z(\cat{C})}$ is the morphism space
$Z(\cat{C})(\alpha , I)$, where $\alpha$ is equipped with the half braiding
from Lemma~\ref{lemmahalfbraiding}.
We may use this to derive the following topological description of sphericality:
 
 \begin{corollary}\label{ValueonS2}
 A pivotal finite tensor category is spherical if and only if the space of conformal blocks of the sphere for its Drinfeld center is one-dimensional.
 \end{corollary}
 
\begin{proof}
	Sphericality is equivalent to $\alpha \cong I$ in $Z(\cat{C})$,
	where $\alpha$ has the half braiding from Lemma~\ref{defhalfbraiding}. 
	
	Therefore, if $\cat{C}$ is spherical, then $Z(\cat{C})(\alpha , I)\cong Z(\cat{C})(I,I)\cong k$ because the unit of $\cat{C}$ is assumed to be simple (and then the unit of $Z(\cat{C})$ is also simple).
	
	 Conversely, if the  conformal block of the sphere $Z(\cat{C})(\alpha , I)$ is one-dimensional, there exists a non-zero map $\alpha \to I$ in $Z(\cat{C})$. But $I$ is simple, and so is $\alpha$ (since it is the image of $I$ under an equivalence). By Schur's Lemma, this implies that $\alpha \cong I$ in $Z(\cat{C})$, thereby making $\cat{C}$ spherical.
	\end{proof}

\begin{remark}
For the Drinfeld center of a pivotal fusion category,
 the dimensions of the vector spaces assigned to the sphere and torus agree with the dimensions
of the corresponding string-net spaces constructed from a not necessarily spherical pivotal fusion category by Runkel~\cite{Non-s-string-nets} (and at least in the Example~\ref{exampleGabelian}, we give some higher genus comparisons). 
This suggests	 that the modular functor
associated to $Z(\cat{C})$ admits a description in terms of string-nets. We hope to come back 
to this in the future.    
\end{remark}

\spaceplease
\begin{example}\label{exampleGabelian}
Let us consider again the ribbon Grothendieck-Verdier category $Z(\Vect_G^{\lambda,d})$
from Example~\ref{exgraded2}. 
Thanks to \cite[Corollary~6.3]{mwansular},
 the space $B_g^{Z(\Vect_G^{\lambda,d})}$ is given by 
 the morphism space $Z(\Vect_G^{\lambda,d})(\alpha , \mathbb{F}^{\otimes g})$
  from $\alpha=(I,\omega^2)$ to the $g$-th monoidal power 
 of the 
  canonical 
coend 
\begin{align}
\mathbb{F}= \int^{X\in Z(\Vect_G^\lambda) } DX\otimes X \ \ .
\end{align}
If we work over $\mathbb{C}$, the category
$Z(\Vect_G^{\lambda,d})$ is semisimple. 
The coend now reduces a finite direct $\mathbb{F} = \alpha \otimes (\bigoplus_{s} s^{\vee}\otimes s )$, where the sum runs over the
isomorphism classes of simple objects of $Z(\Vect_G^{\lambda,d})$. This can be used to compute the dimensions of the spaces of conformal blocks 
via \cite[Example~6.4]{mwansular}. Simplifications arise if we assume that $G$ is abelian and choose $\lambda$ trivial, i.e.\ $\lambda=1$. 
Then all the simple objects of $Z(\Vect_G^{1,d})$ are invertible
and hence $\mathbb{F}\cong \alpha\otimes I^{\oplus |G|^2}$
(note that $|G|^2$ is the
number of simple objects of $Z(\Vect_G^{1,d})$). 
Now
\begin{align}
B_g^{Z(\Vect_G^{1,d})}\cong Z(\Vect_G^{\lambda,d})(\alpha , \mathbb{F}^{\otimes g})\cong \mathbb{C}^{|G|^{2g}	}\otimes  Z(\Vect_G^{\lambda,d})(\alpha,\alpha^{\otimes g}) \ . 
\end{align}
In our example, every invertible object in $Z(\Vect_G^{\lambda,d})$ is simple, so $Z(\Vect_G^{\lambda,d})(\alpha,\alpha^{\otimes g})$ is one-dimensional if $\alpha \cong \alpha^{\otimes g}$ and zero otherwise. 
Since $\alpha \cong \alpha^{\otimes g}$ is equivalent to $d^{2(g-1)}=1$, we find with the Euler characteristic
$\chi(\Sigma_g)=2-2g$ of the surface $\Sigma_g$,
\begin{align}
\dim B_g^{Z(\Vect_G^{1,d})} = \begin{cases}
|G|^{2g} & , \text{ if } d^{\chi(\Sigma_g)} = 1 \ ,  \\
0 &, \text{ if } d^{\chi(\Sigma_g)}\neq 1 \ . 
\end{cases} 
\end{align}
For $G=\Z_n$ and $\lambda=1$, this  agrees with the dimensions of the string-net space computed in~\cite[Proposition~1.3]{Non-s-string-nets}. 
In summary, one observes the following: In the spherical case $d^2=1$, we find in genus $g$ a $|G|^{2g}$-dimensional vector space. 
In the non-spherical case $d^2 \neq 1$, the vector space associated to the sphere becomes zero, but we still have non-trivial higher genus contributions whenever higher even powers of $d$ become trivial. More precisely, this happens at those surfaces $\Sigma_g$ for which $d^{\chi(\Sigma_g)} = 1$. One may interpret this condition as a higher genus version of sphericality. It is always fulfilled for the torus because of $\chi(\Sigma_g)=0$; in fact, that must also be true by Corollary~\ref{cortorus}.
\end{example}

	Example~\ref{exampleGabelian} is of course a tractable one; it is in particular semisimple. One may 
	ask for more complicated examples, but actually with a little more effort, it is not too hard to directly write down the general case with the tools from \cite[Section~6]{mwansular}.
	We refrain from spelling out the details and just give the result:
\begin{corollary}\label{corconfblocks}	If $\cat{C}$ is a pivotal finite tensor category, the space 
	of conformal blocks in genus $g$  for the Drinfeld center $Z(\cat{C})$ is
	given by
	\begin{align}
B_g^{Z(\cat{C}) } \cong	Z(\cat{C})\left(\alpha^{\otimes (g-1)} , \int^{X\in  Z(\cat{C}) }     X^\vee \otimes X\right) \ . \label{eqnconfblockZ}
	\end{align}\end{corollary}

\begin{remark}
	The expression on the right hand side of~\eqref{eqnconfblockZ} can be interpreted as the 
	`$\alpha^{\otimes (g-1)}$-twisted'
	invariants of the usual canonical coend $\int^{X\in  Z(\cat{C}) }     X^\vee \otimes X$ of the Drinfeld
	center. 
		As in Example~\ref{exampleGabelian}, the condition $\alpha^{\otimes (g-1)}\cong I$ in $Z(\cat{C})$ may be understood as a higher genus version of sphericality (with the case $g=0$ being the original one). 
	The torus is once again the only surface for which this condition is always empty.
	If $\cat{C}$ is given by modules over a pivotal finite-dimensional
	Hopf algebra $H$, the 
	 coend 
	 $\int^{X\in  Z(\cat{C}) }     X^\vee \otimes X$
	 is the coadjoint representation of the quantum double $D(H)$.  
	 Then~\eqref{eqnconfblockZ} specializes to
	 \begin{align}
	 	\Hom_{D(H)}\left(    \alpha^{\otimes (g-1)} , D(H)^*_\text{coadj}\right)
	 \ . 
	 \end{align}
 A non-spherical example for $H$ is Sweedler's Hopf algebra, see \cite[Example~3.5.6]{dsps}, even though it is actually trace spherical. 
	\end{remark}

	If in Corollary~\ref{corconfblocks} we set $g=1$, we find
	\begin{align}
		B_1^{Z(\cat{C}) } \cong	Z(\cat{C})\left(I , \int^{X\in  Z(\cat{C}) }     X^\vee \otimes X\right) \ . 
		\end{align}
	\label{Value on T2}

	\begin{corollary}\label{cortorus}
		Let $\cat{C}$ be a pivotal fusion category. The dimension of the vector space $B^{Z(\cat{C})}_1$ 
		is given by the number of simple objects in $Z(\cat{C})$.   
	\end{corollary}
	
	\begin{proof} The proof is essentially a special case of \cite[Example~6.7]{mwansular}:
		If $\cat{C}$ is semisimple, i.e.\ a fusion category, then 
		$ \int^{X\in  Z(\cat{C}) }     X^\vee \otimes X \cong \bigoplus_i X_i^\vee \otimes X_i$ with the sum running over all the simple objects in $Z(\cat{C})$, see  \cite[Corollary~5.1.9]{kl}. This implies the assertion.
		\end{proof}

\spaceplease
\section{Classification of ribbon Grothendieck-Verdier structures on a balanced braided category\label{seccyclicstructures}}
In this final section,
we address a very natural problem that, in particular, we need to solve
to complete the proof of the uniqueness statement in Theorem~\ref{thmribbon}:
Given a ribbon Grothendieck-Verdier category $\cat{A}$ (in $\Lexf$),
we can ask whether the ribbon Grothendieck-Verdier duality that $\cat{A}$ comes equipped with is in fact the \emph{only}
 ribbon Grothendieck-Verdier duality on the underlying balanced braided category of $\cat{A}$. 
In other words, we would like to understand the set $\catf{RibGV}(\cat{A})$ of equivalence classes of ribbon Grothendieck-Verdier structures
 on the underlying balanced braided category of $\cat{A}$. This is a pointed set with the equivalence class of $\cat{A}$ being the base point.

In order to describe
$\catf{RibGV}(\cat{A})$,
we need some preparations:
Let $\cat{A}$ be a ribbon Grothendieck-Verdier category in $\Lexf$.
By the main result of \cite{cyclic} this is exactly a cyclic algebra over the cyclic framed $E_2$-operad. 
For the sake of conciseness, we will not expand here on the relatively involved technical details behind the notion of a cyclic framed $E_2$-algebra in a symmetric monoidal bicategory. For the reader interested in the details, we refer to \cite[Section~2]{cyclic} or, for a short summary, to \cite[Section~2]{mwansular}.
Let us now
introduce  notation
for the following automorphism 2-groups as defined in~\cite[Section~2.4]{cyclic}:
 \begin{itemize}
	\item We denote by $\catf{AUT}_{\framed}(\cat{A})$ the 2-group of automorphisms of $\cat{A}$
	as non-cyclic framed $E_2$-algebra (i.e.\ as balanced braided category). These are braided monoidal autoequivalences of $\cat{A}$ preserving the balancing.
The isomorphism classes of these autoequivalences form a group
\begin{align}
	\catf{Aut}_{\framed}(\cat{A}):=\pi_0(\catf{AUT}_{\framed}(\cat{A})) \ . \end{align}
The fundamental group $\pi_1(\catf{AUT}_{\framed}(\cat{A}),\id_\cat{A})$ is the group of 
monoidal automorphisms of $\id_\cat{A}$ (at this categorical level, there are no further conditions for the braiding and balancing):
\begin{align}
	\pi_1(\catf{AUT}_{\framed}(\cat{A}),\id_\cat{A}) =\catf{Aut}_{\otimes}(\id_\cat{A}) \ . 
	\end{align}

 	\item We denote by $\catf{cAUT}_{\framed}(\cat{A})$ 
  the 2-group of autoequivalences of $\cat{A}$ as \emph{cyclic} framed $E_2$-algebra in $\Lexf$ (i.e.\ as ribbon Grothendieck-Verdier category).
  These are braided monoidal automorphisms preserving the balancing that additionally preserve up to coherent isomorphism (this isomorphism is data) the symmetric non-degenerate pairing $X \boxtimes Y \mapsto \kappa(X,Y):=\cat{A}(DX,Y)$.
  We call these \emph{cyclic} balanced braided autoequivalences.
  Again, the path components of this 2-group form a group
   \begin{align}
  	\catf{cAut}_{\framed}(\cat{A}):=\pi_0(\catf{cAUT}_{\framed}(\cat{A}))\ . \end{align}
  The fundamental group $\pi_1(\catf{cAUT}_{\framed}(\cat{A}),\id_\cat{A})$ is the group of 
  monoidal automorphisms $\lambda$ of $\id_\cat{A}$ that satisfy $\lambda_{DX}=D\lambda_X$ for all $X\in\cat{A}$. Since this is again a compatibility with the cyclic structure, we denote this group by  $\catf{cAut}_{\otimes}(\id_\cat{A})$. Then we have
  \begin{align}
  	\pi_1(\catf{cAUT}_{\framed}(\cat{A}),\id_\cat{A}) =\catf{cAut}_{\otimes}(\id_\cat{A}) \ . 
  \end{align}

\end{itemize}
We will ultimately describe
$\catf{RibGV}(\cat{A})$ via the balanced version 
$Z_2^\catf{bal}(\cat{B})$
of the Müger center of $\cat{A}$ or rather its Picard groupoid $\catf{Pic}(Z_2^\catf{bal}(\cat{A}))  $:
Recall that the \emph{Müger center $Z_2(\cat{B})$}
of a braided monoidal category $\cat{B}$ is the full subcategory of its \emph{transparent objects}, i.e.\
those objects that trivially double braid with every other object. If $Z_2(\cat{B})$ is generated by the monoidal unit of $\cat{B}$ under direct sums, then we call $\cat{B}$ (or rather its braiding)
 \emph{non-degenerate}. 
If $\cat{B}$ is also balanced, we  define the full subcategory $Z_2^\catf{bal}(\cat{B}) \subset Z_2(\cat{B})$ of those transparent 
objects  that have trivial balancing; we call 
$Z_2^\catf{bal}(\cat{B})$
the \emph{balanced Müger center of $\cat{B}$}. 

The \emph{Picard group} $\catf{Pic}(\cat{C})$ of any monoidal category $\cat{C}$
is the group of isomorphism classes of invertible
 objects with the monoidal product as group operation. 

Having established all this notation, we can state the main result of this section:

\begin{theorem}\label{thmexactseq}
	For any ribbon Grothendieck-Verdier category $\cat{A}$ in $\Lexf$, there is a canonical seven-term exact sequence
	\begin{center}
	\begin{equation}\centering
		\hspace*{-3cm}\begin{tikzcd}
			  1\ar[rrr] &&&
		\catf{cAut}_{\otimes}(\id_\cat{A}) 	\ar[rrr]
			&&& \catf{Aut}_{\otimes}(\id_\cat{A})    \ar[out=355,in=175,dllllll]    \\
			\catf{Aut}_\cat{A}(I) 	 \ar[rrr] &&&  \catf{cAut}_{\framed}(\cat{A}) \ar[rrr] &&& \catf{Aut}_{\framed}(\cat{A}) \ar[out=355,in=175,dllllll] \\ \catf{Pic}(Z_2^\catf{bal}(\cat{A}))   \ar[rrr] &&&   \catf{RibGV}(\cat{A}  ) \ar[rrr] &&& 1
		\end{tikzcd} 
	\end{equation}\end{center}
	of groups, except for the last stage where it is an exact sequence of pointed sets. 
\end{theorem}

Instead of directly establishing the exact sequence from Theorem~\ref{thmexactseq} featuring the groups $\catf{cAut}_{\framed}(\cat{A}), \catf{Aut}_{\framed}(\cat{A}) $ and $ \catf{Pic}(Z_2^\catf{bal}(\cat{A}))$,
we will make a conceptually much clearer and simpler
statement about the corresponding 2-groups, i.e.\ about the `categorified' quantities. Let us explain this second slightly more abstract version of the statement and explain how it implies Theorem~\ref{thmexactseq}.
To this end, we introduce the full 2-groupoid
$\catf{RIBGV}(\cat{A})$
of ribbon Grothendieck-Verdier categories in $\Lexf$ whose underlying balanced braided category is  $\cat{A}$.
	This definition is made in such a way that $\catf{RibGV}(\cat{A})=\pi_0(\catf{RIBGV}(\cat{A}))$. 
Moreover, denote by $\catf{PIC}(   Z_2^\catf{bal}(\cat{A})  )$ the Picard 2-group of $Z_2^\catf{bal}(\cat{A})$ whose objects are the invertible objects in $Z_2^\catf{bal}(\cat{A})$ and whose morphisms are just morphisms between these invertible objects. This definition is made in such a way that $\catf{Pic}(   Z_2^\catf{bal}(\cat{A})  )=\pi_0(\catf{PIC}(   Z_2^\catf{bal}(\cat{A})  ))$.

The statement that we now want to prove is a statement about the homotopy fiber of the forgetful map $q: \catf{RIBGV}(\cat{A})\to B \catf{AUT}_{\framed}(\cat{A})$
from $\catf{RIBGV}(\cat{A})$ to the classifying space $B\catf{AUT}_{\framed}(\cat{A})$ of $\catf{AUT}_{\framed}(\cat{A})$:

\begin{theorem}\label{thmfiberseq}
	For any ribbon Grothendieck-Verdier category $\cat{A}$ in $\Lexf$, the homotopy fiber of 
	$q: \catf{RIBGV}(\cat{A})\to B \catf{AUT}_{\framed}(\cat{A})$ over the base point of $B \catf{AUT}_{\framed}(\cat{A})$ is equivalent to the Picard 2-group
	$\catf{PIC}(   Z_2^\catf{bal}(\cat{A})  )$ of the balanced Müger center $Z_2^\catf{bal}(\cat{A})$ of $\cat{A}$, i.e.\
	there is a fiber sequence
	\begin{align}
	\label{eqnfibersequence}	\catf{PIC}(   Z_2^\catf{bal}(\cat{A})  )\to \catf{RIBGV}(\cat{A})\ra{q} B \catf{AUT}_{\framed}(\cat{A}) 
		\end{align}
	of 2-groupoids.
	The map $\catf{PIC}(   Z_2^\catf{bal}(\cat{A})  )\to \catf{RIBGV}(\cat{A})$ sends $I$ to $\cat{A}$.
	\end{theorem}

\begin{remark}\label{rempi2}
	In~\eqref{eqnfibersequence} $\catf{PIC}(   Z_2^\catf{bal}(\cat{A})  )$, which was defined as a groupoid, is seen as a 2-groupoid with trivial 2-isomorphisms.
	In fact, it is directly clear that the homotopy fiber of $\catf{RIBGV}(\cat{A})\to B \catf{AUT}_{\framed}(\cat{A})$ is just a groupoid, even though it is a priori a 2-groupoid. This follows directly from the fact that being a cyclic 2-isomorphism is by definition a \emph{property}, i.e.\ $q$ induces a monomorphism on $\pi_2$. But then $\pi_2$ of the homotopy fiber of $q$ must be trivial.
	\end{remark}

Before proving Theorem~\ref{thmfiberseq},
we briefly say why this gives us immediately Theorem~\ref{thmexactseq}:

\begin{proof}[\slshape Proof of Theorem~\ref{thmexactseq}]
The fiber sequence from Theorem~\ref{thmfiberseq}
produces, after choosing 
base points,
a long exact sequence for the homotopy groups. 
We choose as base points
$\star \in B \catf{AUT}_{\framed}(\cat{A})$ (there is only one), 
$\cat{A} \in \catf{RIBGV}(\cat{A})$
and $I\in \catf{PIC}(   Z_2^\catf{bal}(\cat{A})  )$; these choices are possible because $\catf{PIC}(   Z_2^\catf{bal}(\cat{A})  )\to \catf{RIBGV}(\cat{A})$ sends $I$ to $\cat{A}$ by Theorem~\ref{thmfiberseq}.
The associated 
long exact sequence for homotopy groups
is the precisely the one from Theorem~\ref{thmexactseq}. 
\end{proof}

\spaceplease
\begin{proof}[\slshape Proof of Theorem~\ref{thmfiberseq}]
	We will first prove that the map
	$q: \catf{RIBGV}(\cat{A})\to B \catf{AUT}_{\framed}(\cat{A})$
	 between 2-groupoids is a fibration: To this end, we need to solve two lifting problems (we explain the notation together with the solution after writing the diagrams):
	\begin{equation}\label{eqnaliftingproblem}
	(1) \quad	\begin{tikzcd}
			0 \ar[rr,"\cat{D}"] \ar[dd,"",swap] && \catf{RIBGV}(\cat{A}) \ar[dd,"q"] \\ \\ 
			\Delta^1 \ar[rr,swap,"\varphi"] \ar[rruu,, ,dashed ] & & B \catf{AUT}_{\framed}(\cat{A}) \ , 
		\end{tikzcd}   \qquad (2) \quad \begin{tikzcd}
		\Lambda_1^2  \ar[rrrrr,"\text{$(\widetilde \varphi:\cat{D}\to \cat{E},\widetilde \psi:\cat{E}\to\cat{F})$}"] \ar[dd,"",swap] &&&&& \catf{RIBGV}(\cat{A}) \ar[dd,"q"] \\ \\ 
		\Delta^2 \ar[rrrrr,swap,"\text{$(\varphi,\psi,\chi,\alpha)$}"] \ar[rrrrruu,, ,dashed ] & &&&& B \catf{AUT}_{\framed}(\cat{A}) \ . 
	\end{tikzcd} 
	\end{equation}
	The notation and the solution to the lifting problems is as follows:
	\begin{enumerate}
		\item[(1)] Here $\varphi : \cat{A}\to\cat{A}$ is a balanced braided autoequivalence of $\cat{A}$
		(a 1-automorphism in $ B \catf{AUT}_{\framed}(\cat{A})$), and $\cat{D}$ is a ribbon Grothendieck-Verdier category
		 whose underlying balanced braided category
		  is $\cat{A}$. 
		We solve this lifting problem
		 by transporting the ribbon Grothendieck-Verdier duality of $\cat{D}$ along $\varphi$. 
		 This gives us a ribbon Grothendieck-Verdier category $\cat{D}'$ 
		 that comes with a \emph{cyclic}
		 balanced braided equivalence  $\widetilde \varphi : \cat{D}\to\cat{D}'$ 
		 (a 1-isomorphism in $\catf{RIBGV}(\cat{A})$)
		 mapping to $\varphi$ under $q$.
		 
		\item[(2)] The 2-simplex is labeled with three balanced braided autoequivalences $\varphi, \psi$ and $\chi$ of $\cat{A}$ and a balanced braided isomorphism $\alpha : \psi \varphi \cong \chi$. We denote by $\Lambda_1^2$
		the horn $0\to 1 \to 2$, i.e.\ the 2-simplex with vertices $0$, $1$ and $2$
		 with its interior and the edge from $1$ to $2$ deleted
		(since we are dealing with 2-groupoids, it does matter \emph{which} horn we consider).
		The horn
		$\Lambda_1^2$ is
		 labeled with cyclic equivalences $\widetilde \varphi:\cat{D}\to \cat{E}$ and $\widetilde \psi:\cat{E}\to\cat{F}$ between ribbon Grothendieck-Verdier categories
		 whose underlying balanced braided category is $\cat{A}$ such that $\widetilde \varphi$ and $\widetilde \psi$ map to $\varphi$ and $\psi$, respectively, under $q$. 
		We solve this lifting problem by transporting
		the cyclic structure on $\widetilde \psi \widetilde \varphi$ along $\alpha$ to $\chi$.
		This yields a cyclic equivalence $\widetilde \chi : \cat{D}\to\cat{F}$, and $\alpha$ induces by construction a cyclic isomorphism $\widetilde \psi \widetilde \varphi\cong \widetilde \chi$. This yields the needed lift. 
		\end{enumerate}
	This concludes the proof
	that $q$ is a fibration. Therefore, the strict fiber (and equivalently homotopy fiber) 
	 over the unique 0-cell in $B \catf{AUT}_{\framed}(\cat{A})$ gives us a fiber sequence
	\begin{align}\catf{F} \to \catf{RibGV}(\cat{A})\ra{q} B \catf{AUT}_{\framed}(\cat{A}) \quad \text{with}
	\quad \catf{F}:=q^{-1}(\star)\ .  \label{defFeqn}
	\end{align}
We know
by Remark~\ref{rempi2} already that
 $\catf{F}$ is actually a groupoid.  
	
	 We now have to investigate the fiber $\catf{F}$ in detail: Its objects are ribbon Grothendieck-Verdier structures
	  on the underlying balanced braided category of $\cat{A}$.
	 Let  
	 $D:\cat{A}\to\cat{A}^\opp$ be the ribbon Grothendieck-Verdier duality of $\cat{A}$
	  with dualizing object $K=DI$. Then any different Grothendieck-Verdier structure is, up to equivalence, given by $D'= D\otimes L$,  where $L\in\cat{A}$ is an invertible object.
	 This follows from  \cite[Proposition~2.3~(i)]{bd}, in our dual conventions. 
	 But now we have to ask when this new Grothendieck-Verdier structure is actually ribbon.
	 For this, we need $\theta_{D'X} = D' \theta_X$ for all $X\in \cat{A}$. But
	 $\theta_{D'X} = \theta_{ DX \otimes L} = c_{L,DX} c_{DX,L} ( \theta_{DX} \otimes \theta_{L}  )$ and 
	 $D' \theta_X =  D\theta_X\otimes \id_L = \theta_{DX} \otimes \id _L $.
	 Therefore, the condition amounts to
	 \begin{align}  c_{L,DX} c_{DX,L} ( \theta_{DX} \otimes \theta_{L}  )= 
	 \theta_{DX} \otimes \id _L
	 \quad \text{all}\quad X\in\cat{A} \ . \label{eqnthecondition}  \end{align}
	 Let us first evaluate in the case that $X$ is given by $K=DI$. Then we find thanks to $DK=D^2 I\cong I$ that
	 $\theta_L=\id_L$. Returning to the case of general $X$, we see that $L$ must be in the Müger center. Conversely, if $L$ is in the Müger center and $\theta_L=\id_L$, then~\eqref{eqnthecondition}
	 is satisfied. 
	 This tells us that the objects of the fiber $\catf{F}$ are in fact the invertible objects of the balanced Müger center
	 $Z_2^\catf{bal}(\cat{A})$. 
	 In order to see that $\catf{F}\simeq \catf{PIC}(   Z_2^\catf{bal}(\cat{A})  )$ as groupoids, we have to describe
	  morphisms in the fiber $\catf{F}$. In accordance with~\cite[Definition~2.17]{cyclic}, 
	  morphisms of (cyclic) algebras are defined as (cyclic) algebras in the arrow category of $\Lexf$.
	  After unpacking the definition, we find that a morphism 
	   $(\mathcal{A}, D_1=D\otimes L_1 ) \to (\mathcal{A}, D_2=D\otimes L_2 )$
	   in the fiber $\catf{F}$
	   between two ribbon Grothendieck-Verdier dualities on $\cat{A}$
	    is a morphism of cyclic $\framed$-algebras 
	    such that the underlying morphism of balanced braided monoidal categories is the identity. 
	    The only
	 data not fixed by this requirement is 
	 a natural isomorphism $t$ 
	 \begin{equation}
	 \begin{tikzcd}
	 \mathcal{A}\boxtimes \mathcal{A} \ar[d, "\kappa_1", swap] \ar[rr,"\id \boxtimes \id"] && \mathcal{A}\boxtimes \mathcal{A} \ar[d, "\kappa_2"] \ar[lld, Rightarrow ,"t",swap] \\
	 \Vect \ar[rr,"\id ", swap] && \Vect	 	\ , 
 	\end{tikzcd} 
 	\end{equation}  
 where $\kappa_i$ is the non-degenerate pairing $\kappa_i(X,Y)\coloneqq \mathcal{A}(D_iX,Y)$ on $\mathcal{A}$
 corresponding to $D_i$ with $i=1,2$. 
 The natural isomorphism $t$ has to satisfy one additional condition:
 To formulate this condition, recall that, when expressed in terms of the pairing, the ribbon Grothendieck-Verdier structure
 can be described through an isomorphism $\kappa(I,-\otimes -)\cong \kappa(-,-)$ denoted by $\gamma$
 in~\cite[Definition~5.4 \& Theorem~5.6]{cyclic}, and $t$ has to be compatible with the isomorphisms $\gamma_1$ and $\gamma_2$
 associated to $\kappa_1$ and $\kappa_2$, respectively. 
 This condition is that $\gamma_1$ and $\gamma_2$ define a 2-morphism in the 
  arrow category
  (use again \cite[Theorem~5.6]{cyclic} to see that
   this is the only condition). 
  After spelling out the condition, we see that it  amounts to the commutativity of the following
  diagram:
 \begin{equation}
 \begin{tikzcd}
 \mathcal{A}(D_1I,X\otimes Y) \ar[r,"="] \ar[d, "t_{I,X\otimes Y }",swap] & \mathcal{A}( K\otimes L_1,X\otimes Y) \ar[r,"\cong"] & \mathcal{A}(D_1X , Y) \ar[d, "t_{X, Y }"] \\ 
 \mathcal{A}(D_2I,X\otimes Y) \ar[r,"=",swap] & \mathcal{A}( K\otimes L_2,X\otimes Y) \ar[r,"\cong",swap] & \mathcal{A}(D_2X , Y)\ . 
 \end{tikzcd} \ \ 
 \end{equation}  
We conclude that $t$ is completely determined by its component
$t_{I,-}$, which can be arbitrary. Using the Yoneda lemma this implies that $t$ corresponds to an isomorphism $ K \otimes L_2\longrightarrow  K\otimes L_1$. Since it is an isomorphism, we can as well consider its inverse $ K\otimes L_1 \longrightarrow  K\otimes L_2$. 
It is left to show that every such isomorphism corresponds bijectively to an isomorphism $L_1 \to L_2$. But this follows directly from~\cite[Proposition~2.3~(i)]{bd}.   
\end{proof}

For any balanced braided category $\cat{A}$ in $\Lexf$ (whether it has a ribbon Grothendieck-Verdier structure or not), considering the forgetful map
$q: \catf{RIBGV}(\cat{A})\to B \catf{AUT}_{\framed}(\cat{A})$
makes sense if we understand
$\catf{RIBGV}(\cat{A})$ as the 2-groupoid of all ribbon Grothendieck-Verdier categories whose underlying balanced braided category is $\cat{A}$.
 Of course, now $\catf{RIBGV}(\cat{A})$ could be empty. 
 Then the homotopy fiber of $q$ describes the \emph{space of extensions to a ribbon Grothendieck-Verdier category} --- relative to the already existing balanced braided structure of $\cat{A}$. 
Theorem~\ref{thmfiberseq} now implies:

\begin{corollary}\label{spaceofextcor}
	For any balanced braided category $\cat{A}$ in $\Lexf$,
	the space a ribbon Grothendieck-Verdier extensions relative $\cat{A}$ is empty
	or equivalent to the classifying space $B\catf{PIC}(Z_2^\catf{bal}(\cat{A}))$
	of the Picard groupoid of the balanced braided Müger center of $\cat{A}$. 
	\end{corollary}

\begin{corollary}\label{corcyclic}
		For a balanced braided category $\cat{A}$ in $\Lexf$ with non-degenerate braiding, there is, up to equivalence, at most one extension to a ribbon Grothendieck-Verdier structure. More precisely, the space of choices is empty or $B\Aut_\cat{A}(I)$. The latter is $Bk^\times$ if the unit of $\cat{A}$ is simple.	
\end{corollary}

\begin{proof}
	Non-degeneracy of the braiding means precisely that the Müger center is trivial. But this implies that the balanced Müger center is also trivial and hence 
	$\catf{Pic}(Z_2^\catf{bal}(\cat{A}))=0$.
	Now we use Theorem~\ref{thmexactseq} to obtain the uniqueness of the ribbon Grothendieck-Verdier structure of up equivalence.
	In order to describe the space of choices, we use Corollary~\ref{spaceofextcor}.
\end{proof}

We can now finish the proof of
Theorem~\ref{thmribbon}:
\begin{proof}[\slshape Proof of the uniqueness statement in Theorem~\ref{thmribbon}]\label{endofproof}
	We need to prove that the Drinfeld center $Z(\cat{C})$
	of a pivotal finite tensor category $\cat{C}$, with its usual balanced braided structure, admits at most one ribbon Grothendieck-Verdier structure.
	For this, recall that the braiding of $Z(\cat{C})$ is non-degenerate
	by \cite[Proposition~4.4]{eno-d} and \cite[Theorem~1.1]{shimizumodular}.
	Now the statement follows from Corollary~\ref{corcyclic}.
	\end{proof}

If the duality of a ribbon Grothendieck-Verdier category is actually rigid, the map $\catf{cAut}_{\framed}(\cat{A}) \to \catf{Aut}_{\framed}(\cat{A})$ splits:

\begin{corollary}
	For any finite ribbon category $\cat{A}$ (this includes rigidity), the  extensions
	to a ribbon Grothendieck-Verdier category 
	are in bijection to $\catf{Pic}(Z_2^\catf{bal}(\cat{A}))$, with the neutral element of this Picard group corresponding to the rigid Grothendieck-Verdier structure. 
	For this particular Grothendieck-Verdier structure, 
	$	\catf{cAut}(\cat{A})$ is a split extension of $\catf{RibAut}(\cat{A})$ by $\catf{Aut}_\cat{A}(I)  $; in other words,
	\begin{align}
		\catf{cAut}(\cat{A})\cong \catf{RibAut}(\cat{A}) \ltimes \catf{Aut}_\cat{A}(I)      \label{eqnformulacAut}
	\end{align} by a canonical group isomorphism.
\end{corollary}

\begin{proof}
	Consider $\cat{A}$ as ribbon Grothendieck-Verdier category with its rigid duality. 
	The ribbon autoequivalences of $\cat{A}$ form a 2-group that we denote by $\catf{RIBAUT}(\cat{A})$. We then have
	$\catf{RibAut}(\cat{A})=\pi_0(  \catf{RIBAUT}(\cat{A})   )$.
	A ribbon autoequivalence is in particular a monoidal functor and hence preserves duality
	in a canonical way,
	 and a monoidal automorphism of the identity of $\cat{A}$ is automatically compatible with the rigid duality.
	This provides us with a section
	$\catf{RIBAUT}(\cat{A})=\catf{AUT}_{\framed}(\cat{A})\to \catf{cAUT}_{\framed}(\cat{A})$
	for the map
	$\catf{cAUT}_{\framed}(\cat{A})\to \catf{RIBAUT}(\cat{A})=\catf{AUT}_{\framed}(\cat{A})$. 
	As a result, the long exact sequence from Theorem~\ref{thmexactseq} breaks up, at every level, into short exact sequences.
	This gives us in particular
	\begin{itemize}
		\item the short exact sequence
		\begin{align}	1 \to \catf{Aut}_\cat{A}(I) \to \catf{cAut}_{\framed}(\cat{A}) \to \catf{Aut}_{\framed}(\cat{A}) \to 1
		\end{align} which, by the just mentioned one-sided inverse to the last map, is split and hence gives us~\eqref{eqnformulacAut},
		\item and the isomorphism $\catf{Pic}(Z_2^\catf{bal}(\cat{A}))\cong \catf{CycStruc}(\cat{A}  )$.\end{itemize}
\end{proof}

	\small	\morespaceplease
	\bibliographystyle{alpha}

 \vspace*{1cm}
 \noindent \textsc{Perimeter Institute,  N2L 2Y5 Waterloo, Canada} \\[2ex]
	\noindent \textsc{Institut de Mathématiques de Bourgogne, UMR 5584, CNRS \& Université de Bourgogne, F-21000 Dijon, France}

\end{document}